\documentclass[]{article}
\usepackage{ifpdf}

\newcommand{\Title}{Rayleigh Random Flights on the Poisson line SIRSN}
\usepackage[british]{babel}
\newcommand{\Date}{\today}

\usepackage{calc, xspace}
\usepackage{euscript}

\ifpdf
    \usepackage[usenames,dvipsnames]{xcolor}
\else
    \usepackage[monochrome]{color}
    \usepackage{color}
\fi
\usepackage{graphicx}

\definecolor{linkcolor}{named}{Maroon}
\definecolor{citecolor}{named}{PineGreen}
\definecolor{urlcolor}{named}{RoyalPurple}
\definecolor{okcolor}{named}{OliveGreen}
\definecolor{alertcolor}{named}{BrickRed}

\ifpdf
 \DeclareGraphicsExtensions{.png,.jpg}%
\else
 \DeclareGraphicsExtensions{.eps,.ps,.png,.jpg}%
\fi

\usepackage[totalwidth=480pt,totalheight=680pt]{geometry}

 \usepackage[
 \ifpdf
   pdftex,
   pdfstartview=FitH,
   unicode,
 \fi
 ]{hyperref}
 \hypersetup{
 pdfauthor={Wilfrid S. Kendall},
 pdftitle={\Title},
 pdfsubject={Rayleigh Random Flight processes on a Poisson line SIRSN}
 }
 \hypersetup{pdfstartview=FitH}
 \hypersetup{citecolor=citecolor}
 \hypersetup{linkcolor=linkcolor}
 \hypersetup{urlcolor=urlcolor}
 \hypersetup{colorlinks=true}

\usepackage{WSKmaths}
\newcommand{\length}{\operatorname{len}}
\newcommand{\sdist}{\operatorname{dist}_\text{signed}}
\newcommand{\Diagonal}{\Delta}
\newcommand{\Equiv}{\operatorname{\mathcal{E}}}
\newcommand{\EquivClass}{\operatorname{\mathfrak{E}}}
\renewcommand{\Law}[1]{\operatorname{Law}(#1)}
\newcommand{\Line}{\operatorname{\mathcal{L}}}

\newcommand{\Network}{\operatorname{\mathcal{N}}}
\newcommand{\Route}{\operatorname{\mathcal{R}}}
\newcommand{\RRF}{RRF\xspace}
\newcommand{\SIRSN}{SIRSN\xspace}
\newcommand{\SIRSNRRF}{SIRSN-RRF\xspace}
\newcommand{\Similarity}{S}
\newcommand{\StateSpace}{\left(\Pi\times\Pi\right)\setminus\Diagonal}
\newcommand{\Zenv}{\Psi}
\newcommand{\leftrightsquigarrow}{\leftrightarrow}
\newcommand{\ArcAngle}{%
  \mathord{<\mspace{-9mu}\mathrel{)}\mspace{2mu}}%
}
\newtheorem{remark}{Remark}
\newtheorem{conjecture}{Conjecture}

\begin{document}

 \title{\Title}
 \author{%
\href{http://warwick.ac.uk/wsk}{Wilfrid S. Kendall} 
\\
\href{mailto:w.s.kendall@warwick.ac.uk}%
{\scriptsize w.s.kendall@warwick.ac.uk}%
}
 \date{\Date}
 \maketitle

\begin{abstract}\noindent
We study scale-invariant Rayleigh Random Flights (``\RRF'')
in random environments given by
planar Scale-Invariant Random Spatial Networks (``\SIRSN'')
based on speed-marked Poisson line processes.
A natural one-parameter family of such
\RRF (with scale-invariant dynamics) can be viewed as 
producing ``randomly-broken local geodesics'' on the \SIRSN; we
aim to 
shed some light on
a conjecture that 
a (non-broken) geodesic 
on such a \SIRSN
will never come to a complete stop \emph{en route}.
(If true, 
then all such geodesics can be represented as doubly-infinite sequences of sequentially connected line segments.
This would justify a natural procedure for computing geodesics.)
The family of these \RRF
(``{\SIRSNRRF}''),
is introduced \emph{via} a novel axiomatic theory of
abstract scattering representations for Markov chains
(itself of independent interest).
Palm conditioning
(specifically the Mecke-Slivnyak theorem for Palm probabilities of Poisson point processes)
and 
ideas from the 
ergodic theory of random walks in random environments
are used to show that
at a critical value of the parameter
the speed of the
scale-invariant {\SIRSNRRF} neither diverges to infinity nor tends to zero,
thus supporting the conjecture.

\end{abstract}

\bigskip
\noindent
Keywords and phrases:
\\
\textsc{ 
abstract scattering representation,
critical \SIRSNRRF,
Crofton cell;
delineated scattering process,
Dirichlet forms,
dynamical detailed balance,
environment viewed from particle,
ergodic theorem,
fibre process,
Kesten-Spitzer-Whitman range theorem,
Mecke-Slivnyak theorem,
Metro\-polis-Hastings acceptance ratio,
neighbourhood recurrence,
Palm conditioning,
Poisson line process,
\RRF (Rayleigh Random Flight),
RWRE
(Random Walk in a Random Environment),
\SIRSN
(Scale-invar\-iant random spatial network),
\SIRSNRRF.}



 \section{Introduction}\label{sec:introduction}

 \citet{AldousGanesan-2013} and \citet{Aldous-2012}
 introduced
 the notion
 of 
 \emph{Scale-Invariant Random Spatial Networks} (\SIRSN),
motivated by the now ubiquitous navigational tool of 
online maps
 (\href{https://www.google.com/maps}{Google Maps}, 
 \href{https://www.bing.com/maps/}{Bing Maps}, 
 \href{https://www.openstreetmap.org}{OpenStreetMap}). 
 Informal experiments suggest that
  at normal scales
  the 
  route-finding algorithms
  of these map tools
  exhibit 
  scale-invariance 
 \citep[Section 1.5]{Aldous-2012},
  and the notion 
 of a \SIRSN
 was introduced
 to model this behaviour.
 A \SIRSN is a random mechanism that generates networks built out of 
 almost surely unique
random routes between specified locations,
required 
both 
to deliver scale-invariant statistics and to ensure 
considerable
 route-sharing between different routes.
 
Of course it is 
easy
 to produce random networks 
with
translation- and isotropy-invariant
statistics: 
 the challenge 
 is to find route-finding models which are also statistically invariant under change of scale.
 
 In particular \citet{AldousGanesan-2013} and \citet{Aldous-2012} introduced an elegant construction
 based on speed-marked Poisson lines
 (actually related to the ``random pattern of streets''
 described by \citealp[Plate 105]{Mandelbrot-1977}).
 Significant mathematical effort
 \citep{Kendall-2014c,Kahn-2015}
delivered
rigorous proof that this led to a random map,
 the geodesics of which did indeed provide a model 
 for the \SIRSN mechanism.
 However one issue is still unresolved:
 can the geodesics of this map always be expressed
 as sequentially connected doubly-infinite lists of
 segments from the Poisson lines? Colloquially
 this can be expressed as the conjecture that
 geodesics 
on such a \SIRSN
will never come to a complete stop \emph{en route}.
If this conjecture is true, then it justifies
the natural approximation of geodesics 
using finite-line approximations to the \SIRSN.

Motivated by these considerations,
this paper characterizes and describes 
a natural 
one-parameter family
of random flight processes
on the \SIRSN.
Such a process may be viewed as producing ``randomly-broken
local geodesics''.
We show that there is a critical value of the parameter 
at which the speed of the random process is neigh\-bourhood-recurrent, amounting to evidence in favour of the conjecture.
 
To fix ideas and notation, we summarize the
definition of 
 a general \SIRSN mechanism
 \citep{Aldous-2012}:
 \begin{defn}\label{def:SIRSN}
 A \SIRSN (based on a given probability space \((\Omega, \mathcal{F}, \mathbb{P})\)) 
 is a random mechanism that
 takes as input a set of nodes \(x_1\), \ldots, \(x_n\) in \(\Reals^d\), 
 and outputs a (random) network \(\Network(x_1,\ldots,x_n)=\Network_\omega(x_1,\ldots,x_n)\) 
 composed
 of continuous paths or \emph{routes} 
 \(\Route(x_i,x_j)=\Route_\omega(x_i,x_j)\) connecting all 
 pairs of distinct nodes \(x_i\) and \(x_j\).
 (The explicit dependence on \(\omega\in\Omega\) will 
typically
 be suppressed in the following).
 The connecting route \(\Route_\omega(x,y)\) 
 between
 two specified endpoints \(x\) and \(y\) 
 must be uniquely determined for almost all \(\omega\in\Omega\).
 In addition the
 following \emph{axioms} must be satisfied:
  \begin{enumerate}[label=\({\ref{def:SIRSN}.\arabic*}\)]
  \item\label{def:SIRSN.1} \emph{Similarity-invariant statistics:} 
  For each Euclidean similarity \(\Similarity\) 
  (combined translation, rotation and scaling dilation),
    the networks \(\Network(\Similarity x_1,\ldots,\Similarity x_n)\) and \(\Similarity \Network(x_1,\ldots,x_n)\)
  have the same statistical law.
  \item\label{def:SIRSN.2}\emph{Finite mean length:} Let \(D_1=\length(\Route(x,y))\) be the length of the route \(\Route(x,y)\) between two nodes \(x\) and \(y\) separated by unit distance.
   It is required that the mean \(\Expect{D_1}\)
   of this length be finite.
   \item\label{def:SIRSN.3} The \emph{(Strong) \SIRSN property}: Suppose that \((\Omega,\mathcal{F},\mathbb{P})\) supports independent unit intensity Poisson processes
     \(\Xi_1\), \(\Xi_2\), \ldots which are also independent of the \SIRSN.
    Consider the extended network connecting all points of the \emph{dense} Poisson point process
    \(\widetilde{\Xi}=\bigcup\{\Xi_1, \Xi_2, \ldots\}\). 
    Restrict attention to the ``long-range'' part of the network, containing those portions of connecting paths 
    which are more than distance \(1\) from source or destination, with union given by
    \(\bigcup\{\Route(x,y)\setminus(\ball(x,1)\cup\ball(y,1)):x,y\in\widetilde{\Xi}\}\).
    Viewing this part of the network 
    as a \emph{fibre process} \citep[Chapter 8]{ChiuStoyanKendallMecke-2013},
    it is required that 
    this ``long-range'' fibre process should have
    finite length fraction \(\rho\),
    which is to say, finite mean length per unit area / volume / hyper-volume.
    (\citeauthor{Aldous-2012} uses the term ``edge-intensity'' for the length fraction \(\rho\)).
 
\end{enumerate}
 \end{defn}

\begin{remark}
 Note that \citeauthor{Aldous-2012} defines \SIRSN only in 
 the planar case of \(d=2\).
 Despite the complete absence of intersections between Poisson lines 
 in spaces of dimension \(3\) or higher,
\SIRSN 
 based on Poisson line processes 
 in higher dimensions 
 do in fact exist \citep{Kendall-2014c,Kahn-2015}. 
 Nevertheless,
 this paper 
 focusses on the case \(d=2\);
 our questions (in
 particular the conjecture concerning \(\Pi\)-geodesics discussed below) have trivially negative answers
 for \SIRSN based on Poisson line processes in dimensions \(d\geq3\).
\end{remark}

\begin{remark}
 The notion of a ``dense'' Poisson point process \(\widetilde{\Xi}\) needs careful measure-theoretic interpretation
 \citep{AldousBarlow-1981,Kendall-2000b}: 
 it is used here as a convenient short-hand to refer to
 the union of a countable infinite ensemble of independent unit-intensity Poisson point processes
 \(\Xi_1\), \(\Xi_2\), {\ldots}\,.
\end{remark}

\begin{remark}
 The assertion that \(\Xi_1\), \(\Xi_2\), \ldots are independent of the \SIRSN
 should be interpreted as saying that they are independent of the \(\sigma\)-algebra
 \(\sigma\{\Network(x_1,\ldots,x_n):x_1,\ldots,x_n\in\Reals^d, n=2,3,\ldots\}\),
 viewing the networks \(\Network(x_1,\ldots,x_n)\) as random closed sets (the theory of random closed sets
 is covered for example in \citealp[Chapter 6]{ChiuStoyanKendallMecke-2013}).
\end{remark}

\begin{remark}
Axioms
\ref{def:SIRSN.1}, \ref{def:SIRSN.2}, \ref{def:SIRSN.3}
have strong implications.
For example:
\begin{enumerate}[(a)]
 \item 
 The network obtained by using straight lines for routes
 (thus non-random)
 cannot be a \SIRSN;
almost all pairs of distinct routes have intersections
which are singletons or empty, 
and if the network is used to connect the points of the Poisson point process \(\Xi\) 
then
almost surely any distant point would then be connected to some other distant point of \(\Xi\)
by a straight line passing within \(1/2\) of the origin \(\origin\)
and hence contributing length at least a positive amount (\(\surd{3}=2\surd(1-1/4)\) in dimension \(d=2\))
within unit distance of \(\origin\).
As a consequence,
the intersection of the ``long-range'' fibre process with any bounded open set 
will almost surely have infinite total length, violating Axiom \ref{def:SIRSN.3}
and indeed its weaker variants \ref{def:SIRSN.3'} and \ref{def:SIRSN.3''} discussed below.
%
%
%
 \item Axiom \ref{def:SIRSN.2}
 excludes networks generated by means of 
 coupled
 Brownian bridges. 
 \item The route \(\Route(x,y)\) is almost surely uniquely determined by its endpoints \(x\) and \(y\).
 Nevertheless
 this uniqueness need not (and typically does not) hold simultaneously for all possible inputs.
 \item 
 A related notion of a
 \emph{weak \SIRSN}
 replaces 
 Axiom \ref{def:SIRSN.3} by:
\begin{enumerate}[label=\({\ref{def:SIRSN}.3'}\)]
 \item\label{def:SIRSN.3'} The \emph{Weak \SIRSN property}: 
The infinite network 
  \(\Network(\Xi)=\bigcup\{\Route(x,y):x,y\in\Xi\}\),
 (which connects
 all points of an independent unit intensity Poisson point process \(\Xi\))
  should have finite mean length per unit area / volume / hyper-volume.
\end{enumerate}
   Of course Axiom \ref{def:SIRSN.3} implies Axiom \ref{def:SIRSN.3'}.
 \item A
 still weaker 
 notion is that of
 a \emph{pre-\SIRSN}, further weakening Axiom \ref{def:SIRSN.3'}:
\begin{enumerate}[label=\({\ref{def:SIRSN}.3''}\)]
 \item\label{def:SIRSN.3''}
The \emph{pre-\SIRSN property} \citep{Kendall-2014c}: 
The infinite network \(\Network(\Xi)=\bigcup\{\Route(x,y):x,y\in\Xi\}\), connecting all points of an independent unit intensity Poisson point process \(\Xi\),
should have locally-finite random length measure (the mean length measure need not be locally finite).
\end{enumerate}
   Similarly Axiom \ref{def:SIRSN.3'} implies Axiom \ref{def:SIRSN.3''}.
\end{enumerate}
\end{remark}

\emph{A priori} 
the axioms in Definition \ref{def:SIRSN} might be
mutually exclusive,
in which case no \SIRSN could exist.
\cite{Aldous-2012} 
proposed and rigorously justified a concrete
example of a (planar) \SIRSN, namely the ``binary hierarchy''.
The routes of the network \(\Network(x_1,\ldots,x_n)\)
are constructed as fastest paths 
lying in a dyadic cartesian network marked by varying speeds.
The statistics of 
this network are neither stationary, isotropic, nor scale-invariant;
however all these difficulties are removed by suitable randomization.

\cite{Aldous-2012} also proposed a possible \SIRSN based on a speed-marked improper
planar
Poisson line process \(\Pi\), 
in which the individual routes composing the network \(\Network(x_1,\ldots,x_n)\)
are fastest paths using \(\Pi\) (we call these paths \emph{\(\Pi\)-geodesics}).
This mechanism 
is determined by choice of a parameter \(\gamma>2\):
each line of \(\Pi\)
is marked by a positive speed-limit \(v\geq0\), and \(\Pi\) is defined 
using
a marked planar Poisson line process with intensity measure
\(\nu\) 
given in two equivalent forms by
 \begin{align}\label{eq:plp-sirsn-intensity}
  \nu({{\d}{v}}{{\d}{r}}{{\d}{\theta}}) \quad& =\quad \tfrac{\gamma-1}{2} v^{-\gamma} {{\d}{v}}{{\d}{r}}{{\d}{\theta}}\,,
  \\
  \label{eq:plp-sirsn-intensity-alt}
  \quad& =\quad \tfrac{\gamma-1}{2} v^{-\gamma} \; \sin\phi \;
  {{\d}{v}}{{\d}{s}}{{\d}{\phi}}\,.
 \end{align} 
The first form \eqref{eq:plp-sirsn-intensity}
is based on
parametrization 
of an (unsensed) line \(\Line\) 
using coordinates \(r\in\Reals\) and \(\theta\in[0,\pi)\):
here
\(r\) is the signed distance of \(\Line\) from a reference point often taken to be the origin \(\origin\), 
and \(\theta\) is the angle made by \(\Line\) with a reference line often taken to be the \(x\)-axis.
The second and equivalent form 
\eqref{eq:plp-sirsn-intensity-alt}
is based on a parametrization which
replaces \(r\) by the signed distance \(s\)
\emph{along} the reference line to the intersection with \(\Line\);
the angle \(\phi\) made by \(\Line\)
with the reference line now has to be sine-weighted.
We will use both kinds of parametrization below,
signalled by reference to \eqref{eq:plp-sirsn-intensity} 
or \eqref{eq:plp-sirsn-intensity-alt}.

It is convenient to write \(v(\Line)\) 
for the speed of a line \(\Line\in\Pi\).

Note that
\(\gamma>1\) is required if 
all lines \(\Line\) of speed \(v(\Line)\geq v_0>0\)
taken from such a 
speed-marked Poisson line process 
are to form a \emph{proper} (non-speed-marked) Poisson line process \(\Pi_{\geq v_0}\)
of finite intensity.
The factor \(\tfrac{\gamma-1}{2}\) in \eqref{eq:plp-sirsn-intensity} is a convenient normalization,
chosen so 
that 
\(\Pi_{\geq1}\)
(without speed-marks)
forms a unit intensity Poisson line process.
Routes 
of the \SIRSN
are fastest-possible Lipschitz paths whose almost-everywhere-defined velocities integrate 
the highly singular orientation field provided by \(\Pi\) 
and obey the speed limits given by the speed-marks \(v\).
If \(\gamma>2\) then \(\Pi\) can be used to define a random metric space on \(\Reals^2\):
the random metric is given by the time spent travelling from one point to another by the fastest route;
and this is indeed
a \SIRSN 
(proof is a combination of
\citealp{Kendall-2014c} and \citealp{Kahn-2015}).

The Poisson line process
model for a \SIRSN
has the advantage of being intrinsically stationary and isotropic, with no need for extra randomization;
this follows because
the intensity measure \eqref{eq:plp-sirsn-intensity}
is invariant under Euclidean isometries of
the underlying plane \(\Reals^2\). 
Moreover the scaling transformation
\begin{align}\label{eq:SIRSN-scaling}
 r \quad&\mapsto\quad a \;\overline{r}\,,\\
 v \quad&\mapsto\quad a^{\frac{1}{\gamma-1}}\;\overline{v}\,,\nonumber
\end{align}
also leaves both \eqref{eq:plp-sirsn-intensity} 
and the equivalent \eqref{eq:plp-sirsn-intensity-alt} invariant.
Consequently the distribution of the Poisson line \SIRSN is invariant under scaling if the speed marks are adjusted as indicated in \eqref{eq:SIRSN-scaling}.

%
%


As noted above, the following
conjecture on Poisson line \SIRSN remains open.
\begin{conjecture}\label{conj:complete-stop}
 Given a \SIRSN generated by a planar Poisson line process \(\Pi\), consider a \emph{\(\Pi\)-geodesic} 
providing the fastest route between two specified points.
It is conjectured that such a \(\Pi\)-geodesic
never comes to a complete halt
strictly between its start and its destination.
\end{conjecture}
This unresolved conjecture is related to various observations
in \cite{Aldous-2012}
concerning singly and doubly infinite geodesics
in general planar \SIRSN: however it emphasizes the 
behaviour of the \(\Pi\)-geodesic along its entire length
rather than at its start- and end-points.
In complete contrast, note that in dimension \(d\geq3\)
all non-trivial paths 
would have to
halt \emph{en route} a great deal,
since paths
 in dimension \(d\geq3\)
 can change from one line to another only by using infinitely
 iterated infinite cascades of intervening lines.
 
It is fairly straightforward to use the methods of 
\citet[Section 4]{Kendall-2014c} to show that a straight-line positive-speed internal portion
of a \(\Pi\)-geodesic
in dimension \(d=2\)
must connect directly to two other
straight-line portions.
However in principle it might still be possible
(albeit implausible)
for a \(\Pi\)-geodesic to contain a point which lies at the start and finish
of 
two successive
infinite sequences of 
sequentially connected straight-line portions whose speeds decay to zero
near that point.

One consequence of an affirmative answer to Conjecture 
\ref{conj:complete-stop}
would be that all \(\Pi\)-geodesics 
correspond to
doubly-infinite lists of 
sequentially connected segments of lines from \(\Pi\).
Furthermore \(\Pi\)-geodesics between two points
\(x\) and \(y\) lying on \(\Pi\) would be contained
in the locally finite network of lines of \(\Pi\)
of speed exceeding any small enough \(v>0\);
highly relevant when simulating
\(\Pi\)-geodesics.

We seek
insight concerning
Conjecture \ref{conj:complete-stop} by investigating an associated question
of intrinsic interest:
namely
whether one can build a natural scale-invariant random process on \(\Pi\) which can be viewed as a ``randomly-broken local \(\Pi\)-geodesic'',
and yet which is \emph{speed-neighbourhood-recurrent} (that is to say, neighbourhood recurrence holds for the process given by the
speed of the random process;
so that this speed neither tends to zero nor drifts off to infinity, but returns 
at arbitrarily large times to a neighbourhood of the original speed).
Failure to construct a natural random
process of this form would reasonably count as evidence against the conjecture.

The study of random processes on \SIRSN
is also prompted by the widespread
study of natural random processes on a random structure
(compare for example
the study of Liouville diffusions for Brownian maps and associated structures:
\citealp{Berestycki-2013,GarbanRhodesVargas-2013}).
Such random processes can be used to express a natural geometry for the structure.
For example, in a different context,
the Riemannian geometry expressed by a diffusion
has been
used to describe those smooth elliptic diffusions which admit Markovian maximal couplings
\citep{BanerjeeKendall-2014}.

There are various options
for
defining random processes on a planar
 Poisson line \SIRSN \(\Pi\):
\begin{enumerate}[1:]
 \item 
 a conventional random walk on the plane, independent of the \SIRSN, and connect successive random walk locations by \(\Pi\)-geodesic interpolation.
 However this construction is only weakly linked to the \SIRSN structure of \(\Pi\), and 
 yields
 a random process 
 for which speeds are trivially always revisiting zero,
 since almost surely each random walk location would miss all the lines of \(\Pi\);
 \item 
 construction of Brownian motion on the line structure of \(\Pi\), using a 
 generalization of Walsh or ``spider'' Brownian motion
 \citep{BarlowPitmanYor-1989} to describe the way in which the Brownian motion switches between lines of different speeds.
 However such constructions 
 are not easily related
 to the speed structure of \(\Pi\),
 except indirectly
 by relating Brownian diffusion rate to \(\Pi\) speed marks.
 \item we choose instead to adapt the notion of a Rayleigh Random Flight (\RRF: 
 introduced in \citealp[and associated correspondence]{Pearson-1905}).
 Implementation
 is scale-invariant
 using the following inductive construction:
 proceed at top speed along a chosen line, 
 switch to intersecting lines in a manner controlled by relative speeds,
 (requiring switches to faster lines always to occur),
 and choose the new direction of movement equiprobably 
 from
 the two directions along the new line.
\end{enumerate}

The
\RRF construction
is a planar version of the one-dimensional scattering processes
studied by \cite{KendallWestcott-1987},
in which coupling is used 
to prove limits of Brownian type for an inhomogeneous random scattering process on the line.
Such scattering processes also
arise naturally in statistical mechanics 
(see for example \citealp[chapter 11]{McKean-2014}).
In the \SIRSN context,
interest lies in whether it is possible to choose parameters for a scale-invariant \RRF on a Poisson line \SIRSN
(essentially, to determine the probability of switching when encountering an intersection)
such that the speed of movement of the \RRF particle forms a neighbourhood-recurrent random process,
neither diverging to infinite speed nor converging to zero speed.
The resulting \SIRSNRRF can be viewed as a ``randomly-broken local \(\Pi\)-geodesic'', so if
neighbourhood-recurrence of speed can be obtained by
a natural choice of parameters
then this 
supports
the conjecture that \(\Pi\)-geodesics do not 
halt \emph{en route}.

We ease the task of 
describing constructions of such \SIRSNRRF by assuming
that each line of the \SIRSN is additionally furnished
with a random choice of direction.

The current section has explained the
rationale and the mathematical content of 
the notion of a \SIRSN, and has
motivated the study of \SIRSNRRF
by relating the possibility of speed-neighborhood recurrence of \SIRSNRRF to the question
of whether \SIRSN \(\Pi\)-geodesics can contain
interior points at which they come to a complete stop.
Section \ref{sec:ASP}
then introduces 
concepts which are
helpful in analysing \RRF on \SIRSN, for which
possible switching points form countable dense subsets on
each line of the \SIRSN. 
The complexity of this situation
is usefully addressed by 
taking an axiomatic approach.
We consider an
\emph{abstract scattering representation} (Definition \ref{def:abstract-s-p});
namely an algebraic representation of non-lazy 
discrete-state-space Markov chains 
(chains that have no chance of not moving)
in terms of \emph{transmission probabilities}
(intuitively, the probability of arriving at a state but not necessarily stopping and changing direction there) and
\emph{scattering probabilities}
(intuitively the probability of stopping and changing direction at a state
given that the process arrives there).
In the context of a Poisson line \SIRSN \(\Pi\), the states of the chain
are ordered distinct pairs of lines, corresponding to points at which the \RRF switches from one line to another; so the state-space can be written as
\(\StateSpace\) where \(\Diagonal\) is the diagonal of \(\Pi\times\Pi\).

All non-lazy Markov chains admit abstract scattering representations: 
the presence of an involution
\(a\mapsto\widetilde{a}\)
of the state-space (corresponding to
reversal of direction of travel in our application) permits the state-space
to be broken up into \emph{scattering classes}
\(\Equiv\)
(eventually corresponding to the lines of \(\Pi\)),
and a suitably compatible 
total ordering for each scattering class
(see Definition \ref{def:delineated};
in the \RRF case, a selection for each line of
one of the two possible linear orderings)
then permits the 
transition probabilities to be expressed
purely in terms of the scattering probabilities
and 
probabilities \(\omega_{a,\pm}\)
of initial binary choices of direction within
the relevant scattering class 
(see Theorem \ref{thm:delineated};
and note that it is at this stage that it pays to assume preferred directions for the lines of \(\Pi\)).
If furthermore the involution leads to
dynamical detailed balance 
with respect to a given invariant measure \(\pi\) 
on \(\Pi\)
(Definition \ref{def:reversable})
and the scattering representation
is unbiased, in the sense 
given in that definition,
that 
then the scattering probabilities themselves,
and \(\pi\),
are necessarily defined in terms of ratios of 
prescribed functions \(\kappa(\Equiv)\) of
equivalence classes (Theorem \ref{thm:mh-ratio}).
This leads to 
a
highly desirable
conclusion: the stochastic dynamics of a
dynamically reversible \RRF on a \SIRSN \(\Pi\)
can be defined 
using
prescribed a scattering class function \(\kappa(\Equiv)\)
(where each \(\Equiv\) is actually a line \(\Line\)
of \(\Pi\)).

Section \ref{sec:SIRSN-RRF} continues the story 
by taking account of
the similarity
symmetries of the \SIRSN \(\Pi\) 
controlled by intensity measure
\(\tfrac{\gamma-1}{2}v^{-\gamma}{\d}v{\d}r{\d}\theta\).
A dynamically reversible \RRF on \(\Pi\) 
is said to have similarity-equivariant dynamics
if scattering probabilities and ratios of 
evaluations of \(\pi\)
(considered as functions of \(\Pi\))
are similarity-invariant, while
\(\pi\) itself is Euclidean-invariant.
Palm distribution theory
(specifically the Slivnyak-Mecke Theorem \ref{thm:Slivnyak-Mecke}) and
ergodicity of \(\Pi\) (Theorem \ref{thm:ergodic})
can now be used to argue that ratios of 
\(\kappa(\Line)\) must equal the \(\alpha\)-th
power of ratios of speeds \(v(\Line)\), for a fixed
positive exponent \(\alpha\) (Theorem \ref{thm:similarity-invariance}). If line-changes
are given by a recipe of Metropolis-Hastings form
then Theorem \ref{thm:similarity-invariance}
shows that scattering probabilities and \(\pi\)
are determined entirely by line-speeds and the exponent \(\alpha\): moreover \(\alpha>\gamma-1\) is required
if scattering is to be non-degenerate
(i.e: is not to happen immediately after
time \(0\)).

We thus obtain a natural definition of a \SIRSNRRF on the \SIRSN \(\Pi\), parametrized by the exponent \(\alpha\)
(Definition \ref{def:SIRSN-RRF}); moreover
this \SIRSNRRF is then an irreducible 
Markov chain on the state-space
of ordered 
intersections \(\StateSpace\)
when conditioned on \(\Pi\) (Lemma \ref{lem:irreducible}).

Section \ref{sec:environment} considers
the Markov chain given by the \emph{relative environment}
of the \SIRSNRRF, which is to say, the environment viewed from the \RRF 
after using the group of similarities to transform the \RRF
state to be the intersection of a unit-speed line along the
\(x\)-axis (corresponding to the current line of travel) 
and a further variable-speed line 
(corresponding to the previous line of travel)
intersecting the \(x\)-axis at the origin \(\origin\).
Working with Dirichlet forms related to the dynamically reversible \SIRSNRRF, and using the
Slivnyak-Mecke Theorem \ref{thm:Slivnyak-Mecke},
we establish that the
Markov chain given
by the relative environment
has stationary \emph{probability} distribution
given
(Theorem \ref{thm:stationarity}) by 
the independent superposition of
a unit-speed line along the \(x\)-axis
with a line through \(\origin\)
of 
\begin{enumerate}[(a)]
 \item a random log-speed given by a possibly asymmetric Laplace density;
 \item and a random angle \(\phi\) with the unit-speed line 
(with \(\phi\) having sine-weighted density);
 \item together with a copy of \(\Pi\).
\end{enumerate}
%
The density of the log-speed is symmetric exactly 
at the critical value
\(\alpha=2(\gamma-1)\).
The non-ergodic part of Birkhoff's ergodic theorem now allows us to rule
out non-critical \SIRSNRRF, as in these cases the average log-relative speed
has positive chance of converging to a non-zero limit, and thus the log-speed 
must have positive chance of never returning to any bounded interval 
around the initial log-speed.
 
Section \ref{sec:ergodic} 
uses all this
to show that the critical \SIRSNRRF is 
speed-neigh\-bourhood-recurrent.
This is done by establishing that 
the Markov chain given by the relative environment
of the \SIRSNRRF,
when started according to the stationary probability
distribution, is in fact ergodic.
This follows from Theorem \ref{thm:ergodic-l-r-s},
a variation on an argument 
of \cite{Kozlov-1985}, using the ergodicity of \(\Pi\)
(note that the argument works for all \(\alpha>\gamma-1\)).
The main result of this paper, the 
neighborhood-recurrence of the log-speed process
in the critical case 
(Theorem \ref{thm:nbd-recurrence})
now follows from Theorem \ref{thm:KestenSpitzerWhitman},
an adaptation of the classic Kesten-Spitzer-Whitman range
theorem to the case of continuous 
one-dimensional state-space.

The concluding Section \ref{sec:conclusion}
discusses
related results and possibilities for future work.

 \section[RRF and abstract scattering]{Rayleigh random flights (\RRF) and abstract scattering}\label{sec:ASP} 
 
 The first task is to define 
 a suitable family of Rayleigh random flight processes on \(\Pi\)
 with scale-invariant dynamics; 
 a principal criterion for suitability is that the resulting process
 should be amenable to calculation.
 Moreover it is appropriate for 
 the process 
to be able 
 to change direction whenever encountering any one of the dense 
 countable
 set of line-intersections along a given Poisson line.
 It is useful to control the complexity of this set-up 
 by adopting an abstract approach 
 based on general scattering processes. 
 An additional merit of this approach
 is that it permits
 isolation of
 a particular one-parameter family of 
 \emph{discrete-time} Rayleigh random flight processes (\RRF) on \(\Pi\)
 which
 can be naturally described as {\SIRSNRRF}.

 We motivate the definition of abstract scattering
 by first making
 a few remarks about possible (continuous-time) \RRF
 on Poisson line \SIRSN.
 Let \(\Pi\) be an improper speed-marked planar Poisson line process, with intensity measure given by Equation \eqref{eq:plp-sirsn-intensity} above.
Our primary interest is in
the \SIRSN case \(\gamma>2\), although our results extend to the borderline \emph{\SIRSN candidate} (but non-\SIRSN)
case \(\gamma=2\).
 Note that, in case \(\gamma=2\), 
 \(\Pi\) still possesses Euclidean- and scaling-invariance, 
 even though it no longer possesses
 the \SIRSN property.
   A reasonable if informal
   definition of a (continuous-time) Rayleigh random flight \(X\) 
   on \(\Pi\) runs as follows:
   it is a continuous-time process
   living on the set which is the 
   (countable)
   union of the lines of \(\Pi\) 
   (this random dense \(F_\sigma\) Lebesgue null-set is called the \emph{silhouette} in \citealp{Kendall-2014c}:
   it can be understood as the countable union of the random closed sets formed by lines of speed exceeding \(1/m\) for \(m=1,2,\ldots\)).
   The continuous-time Rayleigh random flight process
   travels along the lines of \(\Pi\), 
   moving at the maximum speed permitted by the relevant speed-limits on \(\Pi\), 
   with changes of direction (switching onto different lines) 
   occurring at a carefully defined
   sequence of random Markov times 
   \(0< \tau_1<\tau_2<\ldots\)
   which will be made up of
   some (but by no means all) encounters with intersections of lines of \(\Pi\).
   We will consider only cases
   in which the resulting sequence of random times will in fact
   be almost surely locally finite up to a
   possibly infinite 
   ``explosion time'' which is the accumulation point of the times of direction-change.
In fact if \(\gamma\geq2\) then
the explosion time 
almost surely 
cannot correspond to the path becoming unbounded in finite time.
This is because
(in the terminology of \citealp{Kendall-2014c})
the path of the RRF is a \emph{\(\Pi\)-path},
namely
a locally Lipschitz path on \(\Pi\)
with top speed almost always 
locally bounded above by relevant speed marks. 
When \(\gamma\geq2\)
a comparison argument
\citep[Theorem 2.6]{Kendall-2014c} 
bounds
distances 
travelled by \(\Pi\)-paths begun in a 
specified compact set and travelling for specified time \(T<\infty\).
Note that the case \(\gamma<2\) is \emph{not} 
 convenient for our purposes;
 the results of \citet{Kendall-2014c} 
 then imply that
 it is possible for locally
 Lipschitz \(\Pi\)-paths to obey the \(\Pi\) speed limits
 and yet to diverge to infinity in finite time, resulting in
 sterile questions about failure of 
 stochastic completeness.
%


As noted above,
we facilitate discussion of the direction of travel along lines of \(\Pi\),
by making arbitrary choices of sense of direction
to endow all the lines of \(\Pi\)
with preferred directions.
   
   The resulting
   continuous-time \RRF \(X\) 
   will be piecewise-linear, and so 
   its paths can be required
   to be c\`adl\`ag and right-differentiable.
   Let \(Y\) denote the right-hand time-derivative of \(X\):
   \[
    Y(t) \quad=\quad \lim_{s\downarrow0}\frac{X(t+s)-X(t)}{s}\,.
   \]
   In particular, the speed \(|Y|\) of \(X\) is the maximum permitted 
   on the current line, which is to say that it
   is determined by the speed-mark of the current line:
   \[
    |Y(t)|\quad=\quad v(\Line) \qquad \text{ for }\Line=X(t)+Y(t)\cdot\Reals\,,
   \]
   where \(\Line=X(t)+Y(t)\cdot\Reals\) is always a line of \(\Pi\) when \(Y(t)\neq0\), and \(v(\Line)\) is the speed-limit of \(\Line\).

   It is convenient to consider the \emph{augmented} process
   \[
   \left(
   (X(t)+Y(t)\cdot\Reals, |Y(t)|, X(t)+Y(t-)\cdot\Reals, |Y(t-)|)
   \;:\;t\geq0
   \right)\,
   \]
recording both the current unsensed line of travel
\(X(t)+Y(t)\cdot\Reals\) and the 
previous unsensed line \(X(t)+Y(t-)\cdot\Reals\)
as well as the 
corresponding
absolute speeds \(|Y(t)|, |Y(t-)|\). 
It is convenient to omit the actual sense or signed direction of travel;
this augmented process is Markov
conditional on \(\Pi\) even given the augmentation
by \(|Y(t)|\) and the further augmentation by \(|Y(t-)|\) 
(which will facilitate
later discussion of dynamical detailed balance),
this is because knowledge of
 \(|Y(t)|\) and \(|Y(t-)|\) can be obtained from
 knowledge of the
speed-marks of the corresponding lines of \(\Pi\).
We derive a discrete-time \RRF
by sampling 
the augmented process
at the 
times \(\tau_n\) when it switches from one line to another.
This (discrete-time) \RRF
is the main subject of study for this paper.
Letting \(\Line_-(\tau_n)\) be the previous line of travel
and letting \(\Line_0(\tau_n)\) be the current line of travel,
we know that \(X(\tau_n)\) is the unique point in the intersection \(\Line_-(\tau_n)\cap\Line_0(\tau_n)\).
This corresponds to obtaining the \RRF by sampling the 
continuous-time Rayleigh random flight process at the instants of scattering
and just before the choices of
direction of travel on the new line;
we may then consider the \RRF as the sampled process as
\(Z=(Z_n=(\Line_-(\tau_n),\Line_0(\tau_n))\;:\; n\geq1)\)
with state-space 
\begin{equation}\label{eq:state-space}
\StateSpace
\quad=\quad
\{(\Line_-,\Line_0)
    \;:\;\Line_-,\Line_0\in\Pi,\; \Line_-\neq\Line_0\}\,,
 \end{equation}
given by ordered pairs of
(speed-marked)
lines \(\Line_-,\Line_0\in\Pi\),
removing the diagonal set 
\(\Delta=\{(\Line,\Line):\Line\in\Pi\}\)
so that \(\Line_-,\Line_0\) must be distinct.
We repeat for emphasis
that \(Z\) is a Markov chain
when \emph{quenched}, which is to say, when conditioned on 
the random environment given by \(\Pi\).

The process \((Z_n:n\geq1)\) is a 
particular instance of a \emph{generalized scattering process}.
It can be viewed as moving from the intersection \(\Line_-\cap\Line_0\) 
along \(\Line_0\) past further intersections
until it chooses to stop
(is scattered) at a 
new intersection \(\Line_0\cap\Line_+\), 
where it will switch to the new line \(\Line_+\) and continue.
This is a planar variation of the scattering processes discussed for 
example in \cite{KendallWestcott-1987}.
We rise above confusing detail about {\SIRSN}s
by introducing a novel
algebraic representation of general scattering for
discrete state-space Markov chains,
always keeping in mind the motivating example of \RRF 
on a Poisson line \SIRSN.
A further benefit of this abstract approach is that it will later allow us to characterize a natural one-parameter
family of {\RRF}s respecting the symmetries of
the \SIRSN \(\Pi\).
\begin{defn}
\label{def:abstract-s-p}
Consider a non-lazy
 discrete-time countable state-space Markov chain \(Z\),
 (non-lazy, so the transition probability matrix has zeroes on the main diagonal).
 An \emph{abstract scattering representation} for \(Z\)
 expresses
 the one-step transition probabilities
 \(p_{a,b}\) of \(Z\)
 in product form
 \begin{equation*}
  p_{a,b} \quad=\quad \omega_{a,b} s_b \qquad \text{ for all states }a,b\,,
 \end{equation*}
for prescribed
\(s_a\in(0,1]\) and \(\omega_{a,b}\in[0,1]\),
where the \emph{transmission probabilities} \(\omega_{a,b}\) form a
matrix with zeroes on the diagonal and the 
\emph{scattering probabilities} \(s_a\) are all positive.
\end{defn}
 \begin{remark}
 Because of the positivity requirement \(s_a>0\), it follows that
\(p_{a,b}>0\) if and only if \(\omega_{a,b}>0\).
Since the matrix \((\omega_{a,b})\) vanishes on the diagonal,
the same must be true of all such matrices \((p_{a,b})\).
 \end{remark}

The content of this definition is algebraic rather than probabilistic.
 In particular the system of scattering and transmission probabilities is not uniquely defined by the resulting Markov kernel
 (note that the choice of transmission probabilities
 \(\omega_{a,b}=p_{a,b}\)
 and scattering probabilities \(s_a=1\) for all states \(a\) and \(b\)
 always determines an abstract scattering representation, since all \(p_{a,a}\) are required to vanish!)
 and the \(\omega_{a,b}\) and \(s_a\) are described as ``probabilities'' only because all are required to lie in the unit interval \([0,1]\). 
 Indeed the
 system of scattering and transmission probabilities
 need not 
 necessarily
 reflect a specific stochastic mechanism of transmission and reflection
 (though this will eventually be the case for our particular example). 
 In this sense a
 general abstract scattering representation
 is purely formal.
 Interesting examples arise by combining
 scattering probabilities
 with transmission according to a fixed stochastic dynamical system:
 in our case the very simple system of constant-speed
 movement in fixed directions.
The resulting axiomatic approach may offer useful perspectives on questions of statistical survival analysis 
 (see for example \citealp{AndersenBorganGillKeiding-1993}). 
 Scattering processes based on deterministic movement
 also arise in the \emph{ZigZag sampler} in
 Markov chain Monte Carlo theory
 (described for example by \citealp{BierkensFearnheadRoberts-2018}; 
 see also the notion of piecewise-deterministic Markov processes introduced
 by \citealp{Davis-1984}).

 In the context of possible \RRF on a
 \SIRSN or \SIRSN candidate, 
 as noted above,
 the relevant state space is the set of ordered pairs of
 distinct lines \((\Line_-,\Line_0)\)
 (for \(\Line_-, \Line_0\in\Pi\)),
 corresponding to pre- and post-scattering lines.
 The quantity \(\omega_{a,{b}}\) can be interpreted as the 
 probability of getting at least as far as 
 \({b}=(\Line_0,\Line_+)\) along a line \(\Line_0\) from \(a=(\Line_-,\Line_0)\),
 while \(s_{b}\) 
 measures the probability of the process being scattered from 
 the current line \(\Line_0\)
 onto the new line \(\Line_+\).

\begin{remark}\label{rem:HB}
 As an aside,
 we indicate a partial answer to an interesting foundational question: 
 which matrices of probabilities \((\omega_{a,b}:\text{ states }a,b)\) 
 (with zeroes down the diagonal)
 can serve as the matrix of transmission probabilities 
 for an abstract scattering representation of some Markov chain?
 Consider the vectors \(w^{(a)}\) given by \(w^{(a)}_b=\omega_{a,b}\). 
 Suppose these all lie in the Banach space \(\ell^{1}(S)\), so \(\|w^{(a)}\|_1=\sum_b \omega_{a,b}<\infty\) for all states \(a\).
 Let \(C\) be the \(\ell^1\)-closure of the convex hull of the vectors \(t e^{(a)}\), 
 where the \(e^{(a)}\) are canonical basis vectors of \(\ell^1(S)\) and
 \(-\infty < t < 1\).
 Then an application of the Hahn-Banach theorem shows that the 
 matrix
 \((\omega_{a,b}:\text{ states }a,b)\) 
 (with non-negative entries, and zeroes down the diagonal)
 can serve as part of an abstract scattering representation of some Markov chain 
 exactly when
 the \(\ell^1\)-closure of the affine span of the \(w^{(a)}\) (for all states \(a\))
 does not intersect the interior of the set \(C\).
 (However \(\ell^1\)-summability of rows
 \(w^{(a)}\) of the matrix \((\omega_{a,b})\) will not hold in the 
 case of our motivating example.)
\end{remark} 

Consider a 
Markov chain admitting a scattering representation, such that its state-space
admits an involution \(a\mapsto\widetilde{a}\) (with no fixed points). Suppose
further
(to simplify future exposition)
that \(p_{a,\widetilde{a}}=0\) for all states \(a\).
%
%
Then
the state-space can be partitioned into equivalence classes
as follows.
\begin{defn}\label{def:equivalence}
Consider a general abstract scattering representation 
of a Markov chain.
Suppose \(a\mapsto\widetilde{a}\) is an involution on the
state-space with no fixed points
and such that
\(p_{a,\widetilde{a}}=0\) for all states \(a\).
Then the state-space supports an equivalence relation
\(a\sim b\) which is obtained by saturating the relation
\(a\leftrightsquigarrow b\), holding
 if either \(\omega_{a,\widetilde b}>0\) 
 or
 \(\omega_{b,\widetilde a}>0\).
Equivalence classes \(\Equiv\), \(\Equiv'\) 
(which we will refer to as \emph{scattering classes})
are 
said to be \emph{connected} when there exists 
 a finite connecting chain of equivalence classes 
 \(\Equiv=\Equiv_1\), \(\Equiv_2\), \ldots, \(\Equiv_k=\Equiv'\)
 and states \(a_k\in\Equiv_k\) in these classes
 such that \(\widetilde{a}_k\in\Equiv_{k+1}\) for \(k=1,\ldots,k-1\).
\end{defn}

In the case of {\RRF}s on \(\Pi\), 
the relevant
involution is given by transposition of lines:
if \(a=(\Line_-,\Line_0)\) then 
\(\widetilde{a}=(\Line_0,\Line_-)\).
Subject to regularity conditions, scattering classes will then
correspond to the lines of \(\Pi\);
the
scattering class \(\Equiv(\Line)\) corresponding to line \(\Line\) is given by the set of all states 
\((\Line_-,\Line)\)
from which the scattering
process moves off along line \(\Line\);
\begin{equation}
 \Equiv(\Line)\quad=\quad 
 \left\{(\Line_-,\Line): \Line_-\in\Pi\setminus\{\Line\}\right\}\,.
\end{equation}
Note moreover that,
 in the Rayleigh Random Flight application, a state \(a\)
 equivalent to its involution \(\widetilde a\) 
 would correspond to a ``reverse scatterer'', 
 which would be able to reverse the direction of travel of the \RRF.
 These are not present in the simple version studied here, 
 but correspond to the one-dimensional symmetric scatterers considered by \citet{KendallWestcott-1987}.
 Their introduction might lead to 
 ``Brownian-like limits'' for \RRF on \(\Pi\), 
 but we do not pursue this here. In particular we
 would then need to adjust this exposition and definitions
 to account well for
 instances when \(p_{a,\widetilde{a}}\)
 could be positive.

\begin{remark}
 In general \(\omega_{a,\widetilde b}>0\), \(\omega_{b,\widetilde c}>0\) need \emph{not} imply \(\omega_{a,\widetilde c}>0\).
 If \(a \sim c\) then either \(a=c\) or there 
 is a finite sequence of
 states \(a=b_0\), \(b_1\), \ldots \(b_{n-1}\), \(b_n=c\)
 such that, for each successive pair of indices
 \(b_i\) and \(b_{i+1}\),
 either \(\omega_{b_{i},\widetilde b_{i+1}}>0\) 
 or \(\omega_{b_{i+1},\widetilde b_{i}}>0\).
\end{remark}

Given an abstract scattering representation and involution, if the scattering classes can be furnished with total orderings
which are compatible with the scattering representation then 
then we can obtain an attractively simple representation
of the transmission probabilities 
(see Theorem \ref{thm:delineated} below).

\begin{defn}
\label{def:delineated}
 A \emph{delineated scattering process} is a Markov chain which admits an abstract scattering 
 representation together with an involution
 \(a\mapsto\widetilde{a}\),
 and satisfying the following compatibility property:
 each scattering class \(\Equiv\) possesses a total ordering 
 \(\prec\) 
 such that the following holds: if \(a\prec b \prec c\) in \(\Equiv\) then
 \(\omega_{b,\widetilde a}+\omega_{b,\widetilde c}\leq1\)
 and moreover
 \begin{equation}\label{eq:total-order}
  \omega_{a,\widetilde c}\quad\leq\quad \omega_{a,\widetilde b} \left(1-s_{\widetilde b}\right)
 \qquad\text{ if }a\prec b \prec c \text{ or }c\prec b \prec a\,,
 \end{equation}
If it is required to emphasize the r\^oles of the involution \(a\mapsto\widetilde{a}\) 
and the family of scattering classes 
\(\EquivClass=\{\Equiv: \Equiv \text{ a scattering class}\}\)
then we speak of an
\emph{\((a\mapsto\widetilde{a}, \EquivClass)\)-delineated scattering process}.
\end{defn}

Inequality \eqref{eq:total-order} implies that \(\omega_{a,\widetilde b}\) is weakly increasing in \(b\) if 
\(b\) ranges over the scattering class of \(a\) and \(b\prec a\), 
 and weakly decreasing in \(b\) if 
 \(b\) ranges over the scattering class of \(a\) and \(a\prec b\). 
 The requirement \(\omega_{a,\widetilde b}+\omega_{a,\widetilde c}\leq1\) if \(b\prec a \prec c\)
 is suggestive of a scattering mechanism
 that chooses the direction of travel on 
 \(\Equiv\) at random once it is scattered at \(a\in\Equiv\).

 \begin{remark}
 Definition \ref{def:delineated} could be expressed  in
 terms of separation rather than ordering: however 
 the use of orderings permits easier notation.
 \end{remark}

In the case of \RRF on \(\Pi\), 
the required total ordering on a scattering class
is obtained from the natural linear ordering on the corresponding line
(using the arbitrary preferred direction chosen for that line).

In general the presence of an involution \(a\mapsto\widetilde{a}\)
as in Definition \ref{def:delineated},
together with compatible total orderings on scattering classes
\(\Equiv\in\EquivClass\),
makes it possible to write down explicit 
expressions for the transmission probabilities largely in
terms of scattering probabilities.
Suppose that \(a\) lies in the scattering class \(\Equiv\) 
Recall that 
the transition probabilities 
\(p_{a,\widetilde b}=\omega_{a,\widetilde b} s_{\widetilde b}\) form a stochastic matrix, 
moreover \(\omega_{a,\widetilde a}=0\)
(a consequence of the  requirement that \(p_{a,\widetilde{a}}=0\)
together with the requirement that \(s_a>0\)), 
while \(\omega_{a,\widetilde b}=0\) if \(b\not\in\Equiv\)
(following from the definition of
scattering class \(\Equiv\) in Definition \ref{def:equivalence}).
Accordingly,
\[
\sum_{z\in\Equiv\;:\; z\prec a} \omega_{a,\widetilde z} s_{\widetilde z}
\;+\;
\sum_{z\in\Equiv\;:\; a\prec z} \omega_{a,\widetilde z} s_{\widetilde z} \quad=\quad
 \sum_{z\in\Equiv\;:\; z\neq a} \omega_{a,\widetilde z} s_{\widetilde z} 
 \quad=\quad
 \sum_{z} p_{a,\widetilde z}\quad=\quad 1\,.
\]
This permits us to represent the statistical behaviour of a delineated scattering process solely
in terms of the scattering probabilities and of
limiting transmission probabilities \(\omega_{a,+}\)
 and
 \(\omega_{a,-}\) for \(a\in S\).
 \begin{thm}\label{thm:delineated}
  For 
  an \((a\mapsto\widetilde{a}, \EquivClass)\)-delineated scattering process, 
  if 
  states \(a\) and \(b\) lie in the same scattering class 
  \(\Equiv\in\EquivClass\) then
  \begin{align}
   p_{a,\widetilde b} \quad&=\quad 
   \omega_{a, \widetilde b} s_{\widetilde b} \quad=\quad
   \begin{cases}
   \omega_{a,+} \left(\prod_{z\in\Equiv\;:\; a\prec z\prec b}(1-s_{\widetilde z})\right) s_{\widetilde b}
   & \text{ if }a\prec b\,,
   \\
   \omega_{a,-} \left(\prod_{z\in\Equiv\;:\; b\prec z\prec a}(1-s_{\widetilde z})\right) s_{\widetilde b}
   & \text{ if }b\prec a\,.
   \end{cases}
   \label{eq:concrete}
  \end{align}
Here
 \begin{align}
\omega_{a,+} \quad&=\quad 
\sup\{\omega_{a,\widetilde{b}}\;:\;
b\in\Equiv,\, a\prec b\}
\quad\leq\quad1
\,,
\nonumber
\\
\omega_{a,-} \quad&=\quad 
\sup\{\omega_{a,\widetilde{b}}\;:\;
b\in\Equiv,\, b\prec a\}
\quad\leq\quad1\,.
\label{eq:limits}
 \end{align}
  Moreover
  \(\omega_{a,-}+\omega_{a,+}=1\), 
  so that \(\omega_{a,-}\), \(\omega_{a,+}\) may be interpreted as 
  the conditional probabilities of scattering in \(-\) and \(+\) directions,
  given scattering has occurred,
  while \(\omega_{a,+}\prod_{z\in\Equiv\;:\; a\prec z}(1-s_{\widetilde z})=\omega_{a,-}\prod_{z\in\Equiv\;:\; z\prec a}(1-s_{\widetilde z})=0\),
  so that eventual scattering occurs with probability \(1\)
  whichever direction is taken.
 \end{thm}
 \begin{remark}
 In Equation \eqref{eq:limits} we may write
  \(\omega_{a,+} =\lim_{b\downarrow a}\omega_{a,\widetilde{b}}\) as a monotonely increasing limit, and similarly 
  \(\omega_{a,-} =\lim_{b\uparrow a}\omega_{a,\widetilde{b}}\).
 \end{remark}
\begin{proof}
 By the symmetry
 between \(\prec\) and \(\succ\) in 
 Definition \ref{def:delineated},
 it is sufficient to deal with the case of
 states \(a\prec b\), 
 with \(a\) and \(b\) 
 lying in the same scattering class \(\Equiv\). 
 Consider the following multiplicative relationship (an inductive consequence of Definition \ref{def:delineated}):
 if \(a\prec u \prec v\prec \ldots \prec z\prec  b\) lie in \(\Equiv\) then
 \begin{equation}\label{eq:chaining}
  \omega_{a,\widetilde b}
  \quad\leq\quad
  \omega_{a,\widetilde u}
  \times
  \left(1-s_{\widetilde u}\right)\left(1-s_{\widetilde v}\right)\ldots\left(1-s_{\widetilde z}\right)
  \,.
 \end{equation}
 As the 
 ordered 
 chain \(a\prec u \prec v\prec \ldots \prec z\prec  b\) is refined, 
 so the product
 \(\left(1-s_{\widetilde u}\right)\left(1-s_{\widetilde v}\right)\ldots\left(1-s_{\widetilde z}\right)\) decreases.
 Taking the limit over the lattice set of refinements, 
 we obtain an upper bound for \(p_{a,\widetilde b}\) in terms of a (typically infinite) product:
 \begin{equation}\label{eq:inequality1}
  p_{a,\widetilde b}
  \quad\leq\quad
  \omega_{a,+}
  \times
  \left(\prod_{z\in\Equiv\;:\; a\prec z\prec b}(1-s_{\widetilde z})\right)
  s_{\widetilde b}\,.
 \end{equation}
Similarly, when \(b\prec a\), we obtain
 \begin{equation}\label{eq:inequality2}
  p_{a,\widetilde b}\quad\leq\quad
    \omega_{a,-}
  \times
  \left(\prod_{z\in\Equiv\;:\; b\prec z\prec a}(1-s_{\widetilde z})\right)
  s_{\widetilde b}
\,.
 \end{equation}
 
 Using simple algebra and then taking limits over successive refinements,
\begin{equation}\label{eq:identity1}
 \sum_{z\in\Equiv\;:\;a\prec z}\left(\prod_{c\in\Equiv\;:\; a\prec c\prec z}(1-s_{\widetilde c})\right)s_{\widetilde z}
 \quad=\quad 
 1 - \prod_{c\in\Equiv\;:\; a\prec c}(1-s_{\widetilde c})
 \quad\leq\quad 1\,,
\end{equation}
with equality holding if and only if \(\sum_{c\in\Equiv\;:\; a\prec c}s_{\widetilde c}\) diverges 
or one of the \(s_{\widetilde c}\) is equal to \(1\).
Likewise
\begin{equation}\label{eq:identity2}
 \sum_{z\in\Equiv\;:\;z\prec a}\left(\prod_{c\in\Equiv\;:\; z\prec c\prec a}(1-s_{\widetilde c})\right)s_{\widetilde z}
 \quad\leq\quad 1\,,
\end{equation}
with equality holding if and only if \(\sum_{c\in\Equiv\;:\; c\prec a}s_{\widetilde c}\) diverges 
or one of the \(s_{\widetilde c}\) is equal to \(1\).

Finally it has been stipulated that \(\omega_{a,\widetilde y}+\omega_{a,\widetilde x}\leq1\)
when \(x\prec a\prec y\) (Definition \ref{def:delineated}) and therefore (using definition \eqref{eq:limits}) it is the case that
\(\omega_{a,-}+\omega_{a,+}\leq1\).
Consequently we can deduce
\begin{multline}\label{eq:grand-finale}
 1\quad=\quad\sum_{z} p_{a,\widetilde z}\quad=\quad 
 \sum_{z\in\Equiv\;:\; z\prec a} \omega_{a,\widetilde z} s_{\widetilde z}
\;+\;
\sum_{z\in\Equiv\;:\; a\prec z} \omega_{a,\widetilde z} s_{\widetilde z}
\\
\;\leq\;
\omega_{a,-} \sum_{z\in\Equiv\;:\;z\prec a}\left(\prod_{c\in\Equiv\;:\; z\prec c\prec a}(1-s_{\widetilde c})\right)s_{\widetilde z}
\;+\;
\omega_{a,+} \sum_{z\in\Equiv\;:\;a\prec z}\left(\prod_{c\in\Equiv\;:\; a\prec c\prec z}(1-s_{\widetilde c})\right)s_{\widetilde z}
\\
\quad\leq\quad
\omega_{a,-}+\omega_{a,+}
\quad\leq\quad
1\,.
\end{multline}
Thus all these inequalities become equalities, 
also forcing equality for 
\eqref{eq:inequality1}
and
\eqref{eq:inequality2},
and 
\eqref{eq:identity1} and \eqref{eq:identity2}.
The proof is completed by noting that
convergence of the infinite sum
\(\sum_{z} p_{a,\widetilde z}=1\)
now
forces
\(\omega_{a,+}\prod_{z\in\Equiv\;:\; a\prec z}(1-s_{\widetilde z})=\omega_{a,-}\prod_{z\in\Equiv\;:\; z\prec a}(1-s_{\widetilde z})=0\).
This is forced because \eqref{eq:grand-finale}
becomes a sequence of equalities:
hence for example either \(\omega_{a,+}=0\)
or \( \sum_{z\in\Equiv\;:\;a\prec z}\left(\prod_{c\in\Equiv\;:\; a\prec c\prec z}(1-s_{\widetilde c})\right)s_{\widetilde z}=1\):
and in the second case 
(by the reasoning following \eqref{eq:identity1}) either
at least one of the constituent
\(s_{\widetilde{z}}\) is equal to \(1\)
or the sum \(\sum_{c\in\Equiv\;:\; a\prec c}s_{\widetilde c}\) diverges
-- and in either case
\(\prod_{z\in\Equiv\;:\; a\prec z}(1-s_{\widetilde z})=0\).
\end{proof}

The following is an immediate corollary of Theorem \ref{thm:delineated}.
\begin{cor}\label{cor:final-divergence}
For a delineated scattering process, if \(a\in\Equiv\) is a state in a
scattering class then 
\begin{align}
 \prod_{c\in\Equiv\;:\;a\prec c}\left(1- s_{\widetilde{c}}\right)
 \quad&=\quad 0 
 \qquad\text{ if } \omega_{a,+}>0\,,
 \nonumber
 \\
 \prod_{c\in\Equiv\;:\;c\prec a} \left(1-s_{\widetilde{c}}\right)
 \quad&=\quad 0 
 \qquad\text{ if } \omega_{a,-}>0\,.
\label{eq:closed-end}
\end{align}
\end{cor}

Nevertheless, the sum of scattering probabilities
has a local summability property.
\begin{cor}\label{cor:local-summability}
Consider a scattering class \(\Equiv\) 
for a delineated scattering process.
If \(a\prec b\in\Equiv\) then
\[
 \sum_{c\in\Equiv\;:\; a\preceq c \preceq b} s_{\widetilde{c}}
 \quad<\quad\infty\,.
\]
\end{cor}
\begin{proof}
Since \(a \prec b\in\Equiv\), there must be 
a finite chain of states
\(a = c_0, c_1, \ldots, c_n, c_{n+1} = b\) 
satisfying
either 
\(\omega_{c_{i-1}\widetilde{c}_i}>0\)
or 
\(\omega_{c_{i}\widetilde{c}_{i-1}}>0\)
for each \(i\).
It follows from \eqref{eq:total-order} 
and the representation provided by Theorem \ref{thm:delineated}
that if \(a \prec c \prec b\) and \(s_c=1\) then 
the finite chain \(c_1\), \ldots, \(c_n\) has to include \(c\).
Hence there can be at most \(n\) distinct states \(c\)
with \(a\prec c \prec b\) and \(s_{\widetilde{c}}=1\).

Moreover, applying \eqref{eq:concrete} of 
Theorem \ref{thm:delineated}
 and the theory of infinite products, 
 if
 one of 
 \(\omega_{c_{i-1}\widetilde{c}_i}>0\)
 or
  \(\omega_{c_{i}\widetilde{c}_{i-1}}>0\),
then
 \(\sum_{c\in\Equiv\;:\; c_{i-1}\prec c \prec c_i} s_{\widetilde{c}}<\infty\)
 if
   \(c_{i-1}\prec c_i\),
while
if \(c_{i}\prec c_{i-1}\)
 then 
 \(\sum_{c\in\Equiv\;:\; c_{i}\prec c \prec c_{i-1}} s_{\widetilde{c}}<\infty\).

Now dissection of the range of summation 
then shows that
 \[
  \sum_{c\in\Equiv\;:\; a\preceq c \preceq b} s_{\widetilde{c}}
  \quad\leq\quad
  \sum_{i=1}^{n+1}\left(
  s_{\widetilde{c}_{i-1}}
  +
  \sum_{c\in\Equiv\;:\;c\text{ lying between }c_{i-1}, c_i} s_{\widetilde{c}}
  \right)
  + s_{\widetilde{b}}
  \quad<\quad \infty\,.
 \]
\end{proof}

We isolate a particular situation 
which will be important later on.
\begin{defn}
 Consider a delineated scattering process.
 If \(\omega_{a,\pm}=\half\) for all states \(a\)
 (for \(\omega_{a,\pm}\) defined as in equation \eqref{eq:limits} 
 of Theorem \ref{thm:delineated})
 then we say that the delineated scattering process is \emph{balanced}.
\end{defn}

Useful structure is added if the abstract scattering
representation concerns a
Markov chain
satisfying dynamical detailed balance
(see for example \citealp[Theorem 1.14 and preceding material]{Kelly-1979}). 
In this case equations of detailed balance relate
(a) an invariant measure \(\pi\) defined on the state-space,
(b) the transition probabilities, and 
(c) an
involution of state-space 
which can be thought of as corresponding to reversal of direction of travel.
\begin{defn}
\label{def:reversable}
Let \((p_{a,b})\) be the transition matrix of a Markov chain
admitting an abstract scattering
representation by \(p_{a,b}=\omega_{a,b} s_b\)
(so in particular it is necessary that \(p_{a,a}=0\)).
 The process governed by \(p_{a,b}=\omega_{a,b} s_b\) 
 satisfies \emph{dynamical detailed balance}
 (is \emph{dynamically reversible) with invariant measure 
 \(\pi\)}
 if
 there is a state-space involution \(a \leftrightarrow\widetilde{a}\) 
 with \(p_{a,\widetilde{a}}=0\),
 and \(\pi\) is a non-negative measure
 on the state-space, 
 with involution and measure related to \(p_{a,b}\) as follows:
 \begin{enumerate}[\((a)\)]
  \item\label{item:conjugacy}
  \(\pi_a = \pi_{\widetilde{a}}\) for all states \(a\);
  \item\label{item:reversible}
  \(\pi_a p_{a,\widetilde b}=\pi_a\omega_{a,\widetilde b} s_{\widetilde b} = \pi_{b}\omega_{{b},\widetilde{a}} s_{\widetilde{a}}=\pi_{b}p_{{b},\widetilde{a}}\) 
  for all states \(a\neq b\);
  \item\label{item:non-triviality} \(\pi_a>0\) for all states \(a\).
  \item\label{item:no-self-transmission}
  \(\omega_{a,\widetilde a}=0\) for all states \(a\) (so \(p_{a,\widetilde a}=0\)).
 \end{enumerate}
  Additionally, the abstract scattering representation
  is said to be \emph{unbiased dynamically reversible} if 
  \(\omega_{a,\widetilde b}=\omega_{{b},\widetilde{a}}\) for all states \(a, b\).
\end{defn}
Note that the (typically \(\sigma\)-finite) 
measure \(\pi\) 
is not normalized: therefore the \(\pi_a\) are defined only up to a common multiplicative constant.

 The conditions \ref{item:conjugacy} and \ref{item:reversible} amount to the assertion 
 that the chain is statistically identical to its time-reversal, 
 so long as we also use the involution to reverse the ``direction of travel'' of the chain.

Unbiased dynamical detailed balance 
refers primarily to the scattering process representation rather than
the Markov chain.
Under unbiased dynamical detailed balance,
condition \ref{item:reversible} is equivalent to the following simpler condition
which does not involve the transmission probabilities:
 \begin{enumerate}[label=\((\alph*')\), start=2]
  \item  \label{item:reversible-adapted}  
\(\pi_a s_{\widetilde b} = \pi_{b} s_{\widetilde{a}}\) for all states \(a\neq b\) such that \(\omega_{a,\widetilde b}=\omega_{b,\widetilde a}>0\).
 \end{enumerate}

 \begin{remark}\label{rem:common-involution}
  When we consider a Markov chain which is both an
  \((a\mapsto\widetilde{a}, \EquivClass)\)-delineated scattering process
  and satisfies dynamical detailed balance, 
  then we will suppose the same involution 
  \(a\mapsto\widetilde{a}\)
 is used in the definition of delineation and in the definition of
 dynamical reversibility.
 In this case if the delineated scattering process is balanced then Theorem \ref{thm:delineated} implies that
 it is automatically unbiased as a scattering process satisfying dynamical reversibility. We will describe such a process as a balanced delineated
 reversible scattering process.
 \end{remark}

 \begin{remark}
 It is a consequence of conditions \ref{item:reversible} and \ref{item:non-triviality}
 that \(\omega_{a,\widetilde b}>0\) if and only if \(\omega_{b,\widetilde a}>0\).
 For \(\omega_{a,\widetilde b}>0\) implies \(p_{a,\widetilde b}>0\) (since \(s_{\widetilde b}>0\) for all \(b\)).
 Since \(\pi_a>0\) and \(\pi_b>0\) by condition \ref{item:non-triviality},
 it follows from condition \ref{item:reversible} that \(p_{b,\widetilde a}>0\) and hence \(\omega_{b,\widetilde a}>0\).
\end{remark}

 As noted above, in the context of \RRF
 on \(\Pi\), we always consider the state-space involution 
 supplied by
 \((\Line_-,\Line_0) \longleftrightarrow (\Line_0,\Line_-)\).
 Setting \(a=(\Line_-,\Line_0)\) and \(b=(\Line_0,\Line_+)\)
 for distinct \(\Line_-,\Line_0,\Line_+\in\Pi\), 
 unbiased dynamical reversibility means that
 transmission along \(\Line_0\) from \((\Line_-,\Line_0)\) 
 to \((\Line_0,\Line_+)\)
 has the same probability as transmission in the reverse direction along \(\Line_0\) from \((\Line_+,\Line_0)\) to \((\Line_0,\Line_-)\).

Recall the representation of transition probabilities
\eqref{eq:concrete} for a delineated scattering process.
The choices of \(\omega_{a,\pm}\) 
are rather more constrained than might at first be supposed.
\begin{cor}\label{cor:symmetry1}
Consider a delineated scattering process satisfying 
dynamical reversibility
 (not necessarily balanced or unbiased), and 
 choose states \(a\prec b\prec c\)
with \(\omega_{a,\widetilde c}>0\), equivalently \(\omega_{c,\widetilde a}>0\).
Then \(\omega_{b,+}=\omega_{b,-}=\frac12\).
\end{cor}

\begin{proof}
Without loss of generality, suppose \(a\prec b \prec c\).
 By \eqref{eq:inequality1}, \(\omega_{a,\widetilde c}=\omega_{a,+}\prod_{u\;:\;a \prec u \prec c}(1-s_{\widetilde u})\), and so positivity of
 \(\omega_{c,\widetilde a}\) forces positivity of \(\prod_{u\;:\;a \prec u \prec c}(1-s_{\widetilde u})\).
 Let \(\pi\) be the invariant measure.
 Using \(\pi_a p_{a,\widetilde c}=\pi_c p_{c, \widetilde a}\),
 and representations derived from \eqref{eq:inequality1} and \eqref{eq:inequality2},
 we can deduce that \(\pi_a\omega_{a,+}s_{\widetilde c} = \pi_c\omega_{c,-}s_{\widetilde a}\)\,.
 Write this as
 \[
  \frac{\pi_a}{s_{\widetilde a}} \omega_{a,+} \quad=\quad \frac{\pi_c}{s_{\widetilde c}} \omega_{c,-} \,.
 \]
Likewise it follows that
 \begin{equation*}
  \frac{\pi_a}{s_{\widetilde a}} \omega_{a,+} \quad=\quad \frac{\pi_b}{s_{\widetilde b}} \omega_{b,-} 
  \qquad\text{ and }\qquad
  \frac{\pi_b}{s_{\widetilde b}} \omega_{b,+} \quad=\quad \frac{\pi_c}{s_{\widetilde c}} \omega_{c,-} 
  \,.
 \end{equation*}
We deduce
 \[
  \frac{\pi_b}{s_{\widetilde b}} \omega_{b,-} \quad=\quad \frac{\pi_b}{s_{\widetilde b}} \omega_{b,+} \,,
 \]
and hence (since \(\pi_b/s_{\widetilde b}>0\), and Theorem \ref{thm:delineated} yields \(\omega_{b,+}+\omega_{b,-}=1\)) it follows that \(\omega_{b,+}=\omega_{b,-}=\frac12\).
\end{proof}

It is a consequence that, for all delineated scattering processes 
satisfying dynamical reversibility, the \(\omega_{b,\pm}\) probabilities are all equal to \(\tfrac12\) except perhaps in the special case 
when \(s_{\widetilde b}=1\). 
This conclusion can be extended
to all states in the case of a delineated scattering process,
satisfying unbiased dynamical reversibility,
and
with sufficiently many states \(b\) with
\(s_{\widetilde b}<1\). (This includes the \SIRSNRRF which will 
be defined later.)

\begin{cor}\label{cor:symmetry2}
Consider a delineated scattering process
satisfying unbiased dynamical reversibility, such that if \(\omega_{a,\widetilde c}>0\) then there is a state \(b\) lying between \(a\) and \(c\).
In that case \(\omega_{a,+}=\omega_{a,-}=\frac12\) for all states \(a\)
and so the delineated scattering process is balanced.
\end{cor}
\begin{proof}
First note for any state \(a\) there must be another state \(c\) with \(\omega_{a,\widetilde c}>0\), for otherwise 
all 
feasible
transitions from \(a\) would
actually have zero probability. 
According to the stated conditions there must be a state \(b\) lying between \(a\) and \(c\):
suppose without loss of generality that \(a\prec b\). 
Now \(\omega_{b,-}=\tfrac12\) by Corollary \ref{cor:symmetry1}: moreover 
the representation \eqref{eq:inequality1} applied to \(\omega_{a,\widetilde c}\) implies that \(\prod_{u\;:\;a \prec u \prec b}(1-s_{\widetilde u})<1\).

The unbiasedness condition \ref{item:reversible-adapted} of Definition \ref{def:reversable} asserts that \(\omega_{a,\widetilde b}=\omega_{b,\widetilde a}\),  
together with
the equations \(\omega_{a,\widetilde b}=\omega_{a,+}\prod_{u\;:\;a \prec u \prec b}(1-s_{\widetilde u})\)
and \(\omega_{b,\widetilde a}=\omega_{b,-}\prod_{u\;:\;a \prec u \prec b}(1-s_{\widetilde u})\),
then imply
that 
\(\omega_{a,+}=\omega_{b,-}\). But \(\omega_{b,-}=\tfrac12\) by Corollary \ref{cor:symmetry1}.
The argument is concluded by using \(\omega_{a,+}+\omega_{a,-}=1\), as established in Theorem \ref{thm:delineated}.
\end{proof}

Within a fixed scattering class of the form specified in Definition
\ref{def:equivalence},
an immediate algebraic consequence 
of general unbiased dynamic reversibility
is that
the
equilibrium measure of a state is proportional to the scattering probability
at that state.
This, together with exploitation of delineated
structure as expressed in Theorem \ref{sec:ASP}, will allow
us to describe suitable \RRF efficiently in terms of 
a prescribed positive 
function on scattering classes (Theorem \ref{thm:mh-ratio} below).
\begin{lem}\label{lem:reverse1}
Consider an unbiased dynamically reversible scattering process.
 Suppose \(a\sim b\); if \(\pi\) is the invariant measure
 then
 \begin{equation}\label{eq:reverse1}
  \pi_a / s_{\widetilde a} \quad=\quad \pi_b / s_{\widetilde b}\,.
 \end{equation}
 For a scattering class \(\Equiv\)
 defined as in Definition \ref{def:equivalence}, 
 we write \(\kappa(\Equiv)\) for the common value of \(\pi_a / s_{\widetilde a}\) for \(a\in\Equiv\).
\end{lem}
\begin{proof}
 It suffices to establish \eqref{eq:reverse1} when \(\omega_{a,\widetilde b}>0\), equivalently \(\omega_{b,\widetilde a}>0\).
 But \eqref{eq:reverse1} holds in this case because of condition 
 \ref{item:reversible-adapted} of Definition \ref{def:reversable}.
\end{proof}
Since \(\pi\) is not normalized, the
\(\kappa(\Equiv)\)
 are defined only up to a common multiplicative constant.

\begin{cor}\label{cor:reverse2}
Consider a unbiased dynamically reversible scattering process.
Suppose \(a\in\Equiv_1\) and \(\widetilde a\in\Equiv_2\)
for scattering classes defined as in Definition \ref{def:equivalence} using the involution supplied by dynamical reversibility. Then
\begin{equation}\label{eq:reverse2}
 s_a / s_{\widetilde a} \quad=\quad \kappa(\Equiv_1)/\kappa(\Equiv_2)
\end{equation}
In particular, \(s_a=s_{\widetilde a}\) when \(a\)
is equivalent to its involution (thus, \(a\sim \widetilde a\)).
\end{cor}
\begin{proof}
 By Lemma \ref{lem:reverse1}, \(s_{\widetilde a}=\pi_a/\kappa(\Equiv_1)\) likewise \(s_{a}=\pi_{\widetilde a}/\kappa(\Equiv_2)\).
 The result now follows from condition 
 \ref{item:conjugacy} of Definition \ref{def:reversable}.
\end{proof}

Since \(\pi_a=\pi_{\widetilde a}\), it follows from Lemma \ref{lem:reverse1}
that \(\pi_a\leq\min\{\kappa(\Equiv_1),\kappa(\Equiv_2)\}\) when
\(a\in\Equiv_1\) and \(\widetilde{a}\in\Equiv_2\). 
\begin{defn}\label{def:deficit}
 Consider a unbiased dynamically reversible scattering process.
Suppose \(a\in\Equiv_1\) and \(\widetilde a\in\Equiv_2\)
for scattering classes defined as in Definition \ref{def:equivalence}.
The \emph{deficit} \(\delta_a\) is given by
\[
 \pi_a  \quad=\quad
 (1-\delta_a) \min\{\kappa(\Equiv_1),\kappa(\Equiv_2)\}\,.
\]
\end{defn}
Note that the deficit \(\delta_a=\delta_{\widetilde{a}}\)
is left invariant by the involution, since \(\pi_a=\pi_{\widetilde{a}}\), while the involution 
\(a\mapsto\widetilde{a}\) simply exchanges
the two scattering classes involved with state \(a\).

The following interpretation of \(\kappa(\Equiv)\) clarifies its r\^ole.
\begin{cor}\label{cor:through-prob}
 Consider a balanced delineated reversible scattering process
 with invariant measure \(\pi\).
 Let \(\Equiv\) be a scattering class: if \(a\in\Equiv\) then
 \begin{equation}\label{eq:kappa-meaning}
  \kappa(\Equiv)
  \quad=\quad
  2 \sum_{c\in\Equiv\;:\;c\prec a} 
   \pi_c \omega_{c,\widetilde{a}}\,.
 \end{equation}
Hence \(\kappa(\Equiv)\) can be interpreted as 
measuring (invariantly)
the in-flow of the process arriving at 
the state \(a\in\Equiv\) from the left
side of \(\Equiv\setminus\{a\}\),
(alternatively, from the right side), 
but not necessarily stopping there.
\end{cor}
\begin{proof}
Apply dynamical reversibility:
 \begin{multline*}
\sum_{c\in\Equiv\;:\;c\prec a} 
 \pi_c \omega_{c,\widetilde{a}}s_{\widetilde{a}}
 \quad=\quad
\pi_a \sum_{c\in\Equiv\;:\;c\prec a} 
 \omega_{a,\widetilde{c}}s_{\widetilde{c}} 
 \quad=\quad
 \\
 \pi_a \omega_{a,-}\times
 \sum_{c\in\Equiv\;:\;c\prec a} 
 \left(\prod_{b\in\Equiv\;:\; c\preceq b\prec a}
 \left(1-s_{\widetilde{b}}\right)
 \right)
 s_{\widetilde{c}}
 \quad=\quad \half \pi_a
 \end{multline*}
 where the last step uses balance, and also the fact
 established in the proof of Theorem \ref{thm:delineated},
 that Inequality \eqref{eq:identity2} is in fact an equality.
 Equation \eqref{eq:kappa-meaning} now follows
 by multiplying through by \(2/s_{\widetilde{a}}\). 
\end{proof}
\begin{remark}
 Combining Equation \eqref{eq:kappa-meaning}
 with the equation corresponding to the right side of 
 \(\Equiv\setminus\{a\}\), we also obtain
 \begin{equation}\label{eq:kappa-meaning2}
  \kappa(\Equiv)
  \quad=\quad
  \sum_{c\in\Equiv\;:\;c\neq a} 
  \pi_c \omega_{c,\widetilde{a}}\,;
 \end{equation}
 so \(\kappa(\Equiv)\) measures (invariantly) the in-flow
 of the process arriving at \(a\) from any other state in \(\Equiv\),
 but not necessarily stopping there.
 (This alternate interpretation holds
 even if the delineated
 reversible scattering process is not balanced).
\end{remark}

 We conclude the discussion of abstract scattering processes by noting 
 a converse result: given an involution and a decomposition of
 state-space into candidate scattering classes
 with attached \(\kappa\) values, there
 are simply-specified assignments of scattering probabilities 
 (motivated by the constructions of Metropolis-Hastings Markov chain Monte Carlo
 -- see for example \citealp[Chapter 1]{GilksRichardsonSpiegelhalter-1995})
 which lead to valid delineated scattering processes satisfying 
 unbiased dynamical reversibility. For this, we require
 that all deficits \(\delta_a\) 
 (Definition \ref{def:deficit}) are set identically equal to zero.
 
\begin{thm}\label{thm:mh-ratio}
 Given a state-space \(S\) supporting an involution \(a\mapsto\widetilde{a}\)
 and a decomposition
 \(\EquivClass=\{\Equiv_1,\Equiv_2,\ldots\}\)
 into disjoint subsets \(\Equiv_1\), \(\Equiv_2\), \ldots,
 a positive function \(\kappa:\Equiv\mapsto\kappa(\Equiv)>0\) defined on 
 \(\EquivClass\),
 and a total ordering on each scattering class \(\Equiv\in\EquivClass\), 
 consider the following
 ``zero-deficit''
 assignment of scattering probabilities:
\begin{equation}\label{eq:mh-ratio}
 s_{\widetilde a} \quad=\quad 
 \min\left\{1, \frac{\kappa(\Equiv_2)}{\kappa(\Equiv_1)}\right\} \qquad \text{ when } a\in\Equiv_1 \text{ and }\widetilde{a}\in\Equiv_2\,.
\end{equation}
Suppose 
for convenience
that no scattering class contains either maximal or minimal elements.
Taking
\(\omega_{a,\pm}=\half\) for all states \(a\), we can use the 
scattering probabilities to define transmission probabilities 
as in \eqref{eq:concrete} of Theorem \ref{thm:delineated}.
The resulting scattering representation corresponds to a
(balanced)
\((a\mapsto\widetilde{a},\EquivClass)\)-delineated scattering
process
 precisely when
 \begin{enumerate}[label=\({\ref{thm:mh-ratio}.\arabic*}\)]
  \item \label{condition:locally-finite}
  for each 
  \(\Equiv\in\EquivClass\),
  and for each \(a\prec b\in\Equiv\), 
  \[
   \sum_{c\in\Equiv\;:\; a\prec c\prec b} s_{\widetilde{c}} \quad<\quad\infty\,.
  \]
  \item \label{condition:extremely-divergent}
  for each 
  \(\Equiv\in\EquivClass\),
  and for each \(a\in\Equiv\),
  \[
   \prod_{c\in\Equiv\;:\; c\prec a} 
   \left(1-s_{\widetilde{c}}\right) \quad=\quad0\,,
   \qquad
   \prod_{c\in\Equiv\;:\; a\prec c} 
   \left(1-s_{\widetilde{c}}\right) \quad=\quad0\,.
  \]
 \end{enumerate}
  Finally this process satisfies unbiased dynamical reversibility, 
  and is therefore a balanced delineated reversible scattering process,
  with identically zero deficits and equilibrium
measure given by
 \[
  \pi_a \quad=\quad \min\left\{\kappa(\Equiv_1), \kappa(\Equiv_2)\right\} \qquad \text{ when } a\in\Equiv_1 \text{ and }\widetilde{a}\in\Equiv_2\,,
 \]
\end{thm}
\begin{proof}
Note that forcing all deficits to be set to zero also forces 
\(\pi_a =\min\left\{\kappa(\Equiv_1), \kappa(\Equiv_2)\right\}\)
for \(a\in\Equiv_1\) and \(\widetilde{a}\in\Equiv_2\):
\eqref{eq:mh-ratio} then follows from Lemma \ref{lem:reverse1}.

 The
 necessity of 
 \ref{condition:locally-finite},
 respectively \ref{condition:extremely-divergent},
 follows from Corollary \ref{cor:local-summability},
 respectively
 Corollary \ref{cor:final-divergence}.

Existence of the required balanced delineated scattering process \(Z\) 
can be demonstrated by noting that it is well-defined
using the following
recursive procedure involving a
sequence of fair coin tosses and 
draws from a Uniform\((0,1)\) distribution:
\begin{enumerate}
 \item Suppose \(Z_n=a\in\Equiv\). 
 Choose between the two branches of \(\Equiv\) 
 with probability \(\omega_{a,+}=\half\) for \(\{b\in\Equiv\;:\;a\prec b\}\) and \(\omega_{a,-}=\half\) for \(\{b\in\Equiv\;:\;b\prec a\}\);
 \item Draw \(U_{n+1}\) from a Uniform\((0,1)\) distribution. 
 If \(\{b\in\Equiv\;:\;a\prec b\}\) is chosen, 
 then set \(Z_{n+1}\) to be the smallest \(\widetilde b\) (with \(a\prec b\), \(b\in\Equiv\)) such that
 \[
  \prod_{c\in\Equiv\;:\;a\prec c \preceq b}(1-s_{\widetilde c}) \quad\leq\quad U_{n+1} 
  \quad<\quad \prod_{c\in\Equiv\;:\;a\prec c \prec b}(1-s_{\widetilde c})
 \]
(this uses the fact that \(\prod_{c\in\Equiv\;:\;a\prec c}(1-s_{\widetilde c})=0\), which itself arises from \eqref{eq:closed-end}).
Similarly,
if \(\{b\in\Equiv\;:\;b\prec a\}\) is chosen, 
then set \(Z_{n+1}\) to be the largest \(\widetilde b\) (with \(b\prec a\), \(b\in\Equiv\)) such that
 \[
  \prod_{c\in\Equiv\;:\;b\preceq c \preceq a}(1-s_{\widetilde c}) \quad\leq\quad U_{n+1} 
  \quad<\quad \prod_{c\in\Equiv\;:\;b\prec c \prec a}(1-s_{\widetilde c})\,.
 \]
 \item Increment \(n\) by \(1\) and go to step 1.
\end{enumerate}
 Finally, unbiased dynamical reversibility follows immediately from computations
 verifying the conditions in  Definition \ref{def:reversable}
 and particularly the condition 
 \ref{item:reversible-adapted},
 which can be substituted for 
 condition \ref{item:reversible}
 in the case of unbiased dynamical reversibility.
\end{proof}

Part of the proof of Theorem \ref{thm:mh-ratio}
describes a simulation procedure which it is convenient 
to reference explicitly.
\begin{cor}\label{cor:mh-ratio}
 In the situation of Theorem \ref{thm:mh-ratio},
 the (balanced)
\((a\mapsto\widetilde{a},\EquivClass)\)-delineated scattering
process \(Z\) can be simulated using
\begin{enumerate}
 \item a choice \(Z_0\) of initial position;
 \item a sequence of independent fair coin tosses to determine 
 direction of travel along each successive line;
 \item and independently a sequence of independent draws from
 the Uniform \((0,1)\) distribution, to determine the distance of travel along each successive line, based only on the relevant scattering probabilities.
\end{enumerate}
The \(n^\text{th}\) corresponding
pair of coin toss followed by uniform draw can be thought of as
the innovation for time \(n\) generating the evolution of \(Z\).
\end{cor}

\begin{remark}
 Condition \ref{condition:locally-finite} implies 
 that the pattern of intersections of a scattering class \(\Equiv\) 
 with other scattering classes of higher \(\kappa\) 
 levels
 must be locally finite
 in \(\Equiv\).
 Condition \ref{condition:extremely-divergent}
 holds if the pattern of intersections of a scattering class \(\Equiv\) 
 with other scattering classes of higher \(\kappa\) must be infinite in number to the right or left of any fixed state,
 which will certainly be the case for \SIRSN; but note that
 condition \ref{condition:extremely-divergent} could
 still be satisfied even if this is not the case.
\end{remark}

\begin{remark}
 More general scattering probabilities can be considered
 which introduce non-zero deficits:
\begin{equation}\label{eq:mh-ratio-modified}
 s_{\widetilde a} \quad=\quad 
 \sigma(\tfrac{\kappa(\Equiv_2)}{\kappa(\Equiv_1)})
   \min\left\{1, \frac{\kappa(\Equiv_2)}{\kappa(\Equiv_1)}\right\} 
 \qquad \text{ when } a\in\Equiv_1 \text{ and }\widetilde{a}\in\Equiv_2\,,
\end{equation}
where the positive function \(\sigma\) is required to satisfy 
the inversion symmetry \(\sigma(u)=\sigma(1/u)\) (in order to ensure that
dynamical reversibility holds)
and 
it is further required that
\(s_{\widetilde a}\) and \(s_a\) 
satisfy the contraint of lying in \((0, 1)\).
In the following, we will consider only the
(zero-deficit) case \(\sigma(u)\equiv1\) introduced above, corresponding to 
an acceptance mechanism
of Metropolis-Hastings type. In the next section 
we will see that this is the only similarity-invariant possibility
once we require the scattering process 
to have identically zero deficits.

%
%
 \end{remark}

 \section[Scale-invariant RRF on SIRSN (SIRSN-RRF)]{Scale-invariant \RRF on \SIRSN (\SIRSNRRF)}\label{sec:SIRSN-RRF}
 
{\RRF}s
based on the Poisson line \SIRSN 
or \SIRSN candidate \(\Pi\)
(after sampled at times of switching lines, and quenching by conditioning
on the random environment \(\Pi\))
can be considered 
as \((a\mapsto\widetilde{a},\EquivClass)\)-delineated scattering processes,
where
\begin{enumerate}[(a)]
\item
the state-space
is 
the set
\(\StateSpace\)
of ordered pairs of distinct speed-marked lines from \(\Pi\),
as given by \eqref{eq:state-space};
 \item 
the involution is given by 
\(a=(\Line_-,\Line_0)\mapsto\widetilde{a}=(\Line_0,\Line_-)\),
 \item 
and scattering classes \(\Equiv\in\EquivClass\) 
turn out to be of the form
\(\Equiv=\{(\Line,\Line_1)\;:\; \Line_1\in\Pi\setminus\{\Line\}\}\)
for a corresponding line \(\Line\in\Pi\), furnished with the
total ordering taken from \(\Line\)
(using our specification of preferred direction
for each \(\Line\in\Pi\)).
\end{enumerate}
We say that the lines of \(\Pi\)
are scattering classes for the process.

So we will 
consider {\RRF}s which under quenching are balanced
delineated reversible scattering processes
based on \(\Pi\) as above,
with scattering classes corresponding 
to the lines of \(\Pi\).
The stochastic dynamics can then be specified by defining the invariant
measure \(\pi_a\) and the scattering probability \(s_a\) for all states \(a\). 
We require these to deliver processes on \(\Pi\) which behave well under 
Euclidean symmetry and changes of scale,
leading in due course to the natural family of \SIRSNRRF
of Theorem \ref{thm:similarity-invariance} 
and Definition \ref{def:SIRSN-RRF}. 
Viewing the invariant measure and scattering probabilities as depending not just on location 
but also on all of \(\Pi\), we are led to:
\begin{defn}\label{def:similarity}
 Consider a balanced
delineated reversible scattering process
based on the Poisson line \SIRSN 
or \SIRSN candidate \(\Pi\)
(quenching by conditioning on \(\Pi\)),
with the lines of \(\Pi\) as scattering classes.
This is said to satisfy
\emph{similarity equivariance} if 
\begin{enumerate}
 \item the scattering probability, when viewed as a function
 \(s(\Line_1,\Line_2;\Pi\setminus\{\Line_1,\Line_2\})\)
 of location \((\Line_1,\Line_2)\) and reduced environment
 \(\Pi\setminus\{\Line_1,\Line_2\}\),
 is invariant under the group of similarities (generated by Euclidean
 motions and changes of scale);
 \item the invariant measure, viewed as a function
 \(\pi(\Line_1,\Line_2;\Pi\setminus\{\Line_1,\Line_2\})\)
 of location \((\Line_1,\Line_2)\) and
 \emph{reduced environment} \(\Pi\setminus\{\Line_1,\Line_2\}\),
 is invariant under the Euclidean motion group; moreover the ratio
 \[
  \frac{\pi(\Line_1,\Line_2;\Pi\setminus\{\Line_1,\Line_2\})}%
  {\pi(\Line_3,\Line_4;\Pi\setminus\{\Line_3,\Line_4\})}
 \]
 is invariant under scale-change.
\end{enumerate}
\end{defn}

\begin{remark}
 A simple geometric argument
 shows
 it suffices to check scale-invariance only for the ratios
 \[
   \frac{\pi(\Line_1,\Line_2;\Pi\setminus\{\Line_1,\Line_2\})}%
  {\pi(\Line_1,\Line_3;\Pi\setminus\{\Line_1,\Line_3\})}\,,
 \]
 since \(\pi(\Line_1,\Line_2;\Pi\setminus\{\Line_1,\Line_2\})\)
 is symmetric in \(\Line_1\) and \(\Line_2\)
 as a consequence of dynamical reversibility, 
 and almost surely all lines in \(\Pi\) intersect.
\end{remark}

Also note that it follows from Lemma \ref{lem:reverse1}
that the similarity-equivariance
and Euclidean invariance
properties of Definition \ref{def:SIRSN-RRF}
imply Euclidean-invariance of
\[
 \kappa(\Line_1)\;=\;\kappa(\Line_1;\Pi\setminus\{\Line_1\})
\quad=\quad
  \frac{\pi(\Line_1,\Line_2;\Pi\setminus\{\Line_1,\Line_2\})}%
  {s(\Line_1,\Line_2;\Pi\setminus\{\Line_1,\Line_2\})}\,,
\]
and similarity-invariance for
\begin{equation*}
 \frac{\kappa(\Line_1;\Pi\setminus\{\Line_1\})}
 {\kappa(\Line_2;\Pi\setminus\{\Line_2\})}
\quad=\quad
  \frac{s(\Line_2,\Line_1;\Pi\setminus\{\Line_1,\Line_2\})}%
  {s(\Line_1,\Line_2;\Pi\setminus\{\Line_1,\Line_2\})}\,.
\end{equation*}
Finally Euclidean-invariance for the deficit
(viewed again as a function
 \(\delta(\Line_1,\Line_2;\Pi\setminus\{\Line_1,\Line_2\})\)
 of location \((\Line_1,\Line_2)\) and \(\Pi\))
follows from
\[
 1 - \delta(\Line_1,\Line_2;\Pi\setminus\{\Line_1,\Line_2\})
 \quad=\quad
 \frac{\pi(\Line_1,\Line_2;\Pi\setminus\{\Line_1,\Line_2\})}
 {\min\{\kappa(\Line_1;\Pi\setminus\{\Line_1\}),
 \kappa(\Line_2;\Pi\setminus\{\Line_2\})\}}\,,
\]
as does similarity-invariance for
\[
\frac{1 - \delta(\Line_1,\Line_2;\Pi\setminus\{\Line_1,\Line_2\})}
{1 - \delta(\Line_3,\Line_4;\Pi\setminus\{\Line_3,\Line_4\})}
\quad=\quad
   \frac{\pi(\Line_1,\Line_2;\Pi\setminus\{\Line_1,\Line_2\})}%
  {\pi(\Line_1,\Line_3;\Pi\setminus\{\Line_1,\Line_3\})}
\,.
\]
Consideration of these remarks also proves the 
following converse.
\begin{lem}\label{lem:spec}
  Consider a balanced
delineated reversible scattering process based
on the Poisson line \SIRSN 
or \SIRSN candidate \(\Pi\),
with the lines of \(\Pi\) as scattering classes.
It is said to
satisfy
\emph{similarity-equivariance} if 
\(\kappa(\Line_1;\Pi\setminus\{\Line_1\})\)
and 
\(\delta(\Line_1,\Line_2;\Pi\setminus\{\Line_1,\Line_2\})\)
determine Euclidean-invariant functions,
and the ratios
\[
\frac{\kappa(\Line_1;\Pi\setminus\{\Line_1\})}
 {\kappa(\Line_2;\Pi\setminus\{\Line_2\})}
\qquad
\text{ and}
\qquad
 \frac{1 - \delta(\Line_1,\Line_2;\Pi\setminus\{\Line_1,\Line_2\})}
{1 - \delta(\Line_3,\Line_4;\Pi\setminus\{\Line_3,\Line_4\})}
\]
are scale-invariant.
In this case the stochastic dynamics of the scattering process
are determined by specifying the functions 
\(\kappa(\Line_1;\Pi\setminus\{\Line_1\})\)
and 
\(\delta(\Line_1,\Line_2;\Pi\setminus\{\Line_1,\Line_2\})\).
\end{lem}

Remarkably, the choice of \(\kappa(\Line_1;\Pi\setminus\{\Line_1\})\)
is heavily constrained by similarity-equivariance. This follows from
an ergodic theorem for Poisson line \SIRSN  or \SIRSN candidates \(\Pi\). This in turn requires
the Slivnyak-Mecke theorem
for Palm conditioning of Poisson processes.
\begin{thm}[Slivnyak-Mecke]\label{thm:Slivnyak-Mecke}
 Suppose \(\Pi\) is a Poisson point process on Euclidean
 space \(\Reals^d\)
 with diffuse intensity measure \(\nu\). 
 For any measurable non-negative function \(f\) 
 on \(\Reals^d\),
 \[
 \Expect{\sum_{x\in\Pi}f(x,\Pi\setminus\{x\})}
 \quad=\quad
 \int \Expect{f(x,\Pi)} \nu({{\d}{x}})\,. 
\]
\end{thm}
 The proof of this result
 is described in 
Example 4.3 of 
\citet[Section 4.4.4]{ChiuStoyanKendallMecke-2013}.

We now state and prove the required ergodic theorem.
\begin{thm}\label{thm:ergodic} 
 Let \(\Pi\) be a Poisson line \SIRSN or \SIRSN candidate,
 and let \(\xi(\Line,\Pi)\) be a
 non-negative measurable function of \(\Line\) and
 \(\Pi\) 
 which is invariant
 under Euclidean motion
 applied to the pair \((\Line,\Pi)\)
 of speed-marked line \(\Line\)
 and speed-marked line pattern \(\Pi\). 
 Consider \(\xi(\Line,\Pi\setminus\{\Line\})\) as \(\Line\) varies over the speed-marked line process \(\Pi\):
 this is almost surely a deterministic function of the speed \(v(\Line)\) alone.
\end{thm}

\begin{proof}
It is enough to consider
non-negative bounded measurable \(\xi\), 
since the general result then follows
by consideration of \((n\wedge\xi)\) for increasing \(n\).
For a fixed speed-marked line \(\Line\),
consider the random process 
\(t\mapsto \xi(\Line;T_t\Pi)\),
where \(T_t\) is Euclidean translation parallel to \(\Line\) for \(t\in\Reals\).
The law of \(\Pi\) is translation invariant, so this bounded random process is stationary.

%
%

Now view \(\Pi\) in \(v-s-\phi\) coordinates corresponding to 
\eqref{eq:plp-sirsn-intensity-alt} based on the
fixed speed-marked
line \(\Line\).
In these coordinates the action of \(T_t\) sends
\((v,s,\phi)\) to \((v,s+t,\phi)\). 
Arguing as in the standard proof of the Hewitt-Savage zero-one law (for example \citealp[Theorem 3.15]{Kallenberg-2010}),
the bounded shift-invariant
function \(\xi(\Line;\Pi)\) can be approximated by a 
non-negative bounded measurable
function \(\widetilde{\xi}(\Line;\Pi)\)
depending only on lines of \(\Pi\) intersecting \(\Line\) in a bounded 
interval \(I\). 
Choosing \(t\) such that \(T_tI\) and \(I\) are disjoint,
\(\xi(\Line;\Pi)\) is similarly approximated by the
statistically independent \(\widetilde{\xi}(\Line;T_t\Pi)\). 
It follows that \(\xi(\Line;\Pi)\) is independent of itself
and thus is almost surely a deterministic function
\(c(\Line)\)
of \(\Line\) alone.
Euclidean-invariance of \(\xi\) implies that
\(\xi(\Line;\Pi)=c(\Line)=c(v(\Line))\) depends only on the speed \(v(\Line)\) of \(\Line\).

 Recall that the speed-marked Poisson line process \(\Pi\)
can be viewed as 
a Poisson point process in \(v-r-\theta\) space
(\((0,\infty)\times\operatorname{linespace})\))
 with intensity \(\nu({\d}{\Line})=\nu({\d}{v}{\d}{r}{\d}{\theta})\) given by \eqref{eq:plp-sirsn-intensity}.
So Theorem \ref{thm:Slivnyak-Mecke} applies to 
\(f(\Line,\Pi\setminus\{\Line\})=
\xi(\Line;\Pi\setminus\{\Line\})\Indicator{\Line\in A}\times\Indicator{\Pi\setminus\{\Line\}\in B}\),
for any
compact subset \(A\subset (0,\infty)\times\operatorname{linespace}\)
and any measurable subset \(B\)
of the space of speed-marked patterns inducing locally finite point patterns on \((0,\infty)\times\operatorname{linespace}\).
Thus
\begin{multline*}
  \Expect{\sum_{\Line\in\Pi\cap A}
  \xi(\Line;\Pi\setminus\{\Line\})\Indicator{\Pi\setminus\{\Line\}\in B}}
 \quad=\quad
 \\
 \int_A \Expect{
 \xi(\Line;\Pi)\;;\; \Pi\in B
 } \nu({\d}{\Line}) 
 \quad=\quad
 \Prob{\Pi\in B}\int_A c(v(\Line))
 \nu({\d}{\Line})
\end{multline*}
(finite because we required \(A\) to be a compact subset of
\((0,\infty)\times\operatorname{linespace}\)).

Viewing this as an equality between measures evaluated on 
product sets
\(A\times B\), a \(\Pi\)-system argument allows us to deduce that
 almost surely
\(\xi(\Line;\Pi\setminus\{\Line\})=c(v(\Line))\)
for all \(\Line\in\Pi\).
\end{proof}

As a direct consequence we have the following result,
which characterizes similarity-equivariant
scattering processes 
based on \(\Pi\) as a one-parameter
family
when the deficit 
\(\delta(\Line_1,\Line_2;\Pi\setminus\{\Line_1,\Line_2\})\)
is required to
vanish identically (equivalently if the invariant measure
\(\pi(\Line_1,\Line_2;\Pi\setminus\{\Line_1,\Line_2\})\)
is maximized subject to the specification of 
\(\kappa(\Line;\Pi\setminus\{\Line\})\)).
Note that 
the result requires one
to check non-triviality of the scattering process
(namely, that scattering can not occur instantaneously, and therefore 
that the scattering classes do indeed correspond to entire lines).
\begin{thm}\label{thm:similarity-invariance}
 Consider a balanced delineated reversible scattering process based on \(\Pi\)
 (quenched by conditioning on \(\Pi\)), 
 which has the lines of \(\Pi\) as scattering classes,
 so that its stochastic
 dynamics are specified by \(\kappa(\Line;\Pi\setminus\{\Line\})\)
and \(\delta(\Line_1,\Line_2;\Pi\setminus\{\Line_1,\Line_2\})\).
 Suppose further that the stochastic dynamics
 are similarity-equivariant.
 Then, for some real parameter \(\alpha\),
 \begin{equation}\label{eq:scaling}
  \frac{\kappa(\Line_1;\Pi\setminus\{\Line_1\})}%
  {\kappa(\Line_2;\Pi\setminus\{\Line_2\})}
  \quad=\quad
  \frac{v(\Line_1)^\alpha}{v(\Line_2)^\alpha}\,.
 \end{equation}
 Moreover, if the deficit 
 \(\delta(\Line_1,\Line_2;\Pi\setminus\{\Line_1,\Line_2\})\)
 vanishes identically 
 (so
 that
 \(\pi(\Line_1,\Line_2;\Pi\setminus{\Line_1,\Line_2})\)
 reduces to \(\min\{v(\Line_1)^\alpha,v(\Line_2)^\alpha\}\),
 and additionally \(s(\Line_1,\Line_2;\Pi\setminus{\Line_1,\Line_2})\)
 reduces to 
 \(\min\{1,(v(\Line_2)/v(\Line_1))^\alpha\}\)),
 then \(\alpha>\gamma-1\) is 
 necessary and sufficient for the scattering process 
 to be non-trivial.
\end{thm}
\begin{proof}
 Applying Theorem \ref{thm:ergodic} to 
 the function
 \(\kappa(\Line;\Pi\setminus\{\Line\})\),
 we deduce that we may write 
 \(\kappa(\Line;\Pi\setminus\{\Line\})=\psi(v(\Line))\).
 Now similarity-invariance for the ratio
 \[
  \frac{\kappa(\Line_1;\Pi\setminus\{\Line_1\})}
 {\kappa(\Line_2;\Pi\setminus\{\Line_2\})}
 \]
 implies, for all \(\lambda>0\),
 \[
  \frac{\psi(v)}{\psi(1)}\quad=\quad\frac{\psi(\lambda v)}{\psi(\lambda)}\,.
 \]
But then
a multiplicative form of Cauchy's functional equation
must hold,
\[
 \frac{\psi(\lambda v)}{\psi(1)}\quad=\quad
  \frac{\psi(v)}{\psi(1)}\times  \frac{\psi(\lambda)}{\psi(1)}\,,
\]
and thus (\(\psi\) being measurable) 
there must be a constant \(\alpha\) such that
\[
 \psi(v)\quad=\quad \psi(1) v^\alpha\,.
\]
Equation \eqref{eq:scaling} follows.

Suppose that the deficit vanishes identically, so in particular
\[
s(\Line_1,\Line_2;\Pi\setminus{\Line_1,\Line_2})
\quad=\quad 
\min\left\{1,\frac{v(\Line_2)^\alpha}{v(\Line_1)^\alpha}\right\}\,. 
\]
The resulting process is a non-trivial \RRF exactly when the conditions
of Corollaries \ref{cor:local-summability} and 
\ref{cor:final-divergence} are satisfied.
The condition for Corollary \ref{cor:final-divergence} follows immediately
from the observation that, for any given line \(\Line\), there is an everywhere dense set of intersections with lines of slower speed (in case \(\alpha<0\))
and an infinite unbounded set of intersections with lines of faster speed
(in case \(\alpha\geq0\). 

Consider the condition for 
Corollary \ref{cor:local-summability}, which is required
if 
the lines of \(\Pi\)
do indeed correspond to the scattering classes of the process.
We shall first show that \(\alpha>\gamma-1\) implies
finiteness of all sums of the form
\[
 \Expect{\sum_{\Line\in\Pi} \xi(\Line) \sum_{\dist(\Line',\origin)<A}
 s(\Line,\Line';\Pi\setminus\{\Line,\Line'\})
 }
\]
whenever \(\xi(\Line)\) is non-negative and measurable,
and \(\Expect{\sum_{\Line\in\Pi} \xi(\Line) v(\Line)^{-(\gamma-1)}}<\infty\).
Apply  
Theorem \ref{thm:Slivnyak-Mecke} (Slivnyak-Mecke) twice over,
and expand using the 
line-space coordinates leading to \eqref{eq:plp-sirsn-intensity}:
\begin{multline*}
 \Expect{\sum_{(\Line_-,\Line_0)\in\StateSpace} 
 \xi(\Line_-) \Indicator{|\sdist(\Line_0,\origin)|<R}
\min\left\{1,\frac{v(\Line_0)^\alpha}{v(\Line_-)^\alpha}\right\}
 }
 \\
 \quad=\quad
  \int\int_{|\sdist(\Line_0,\origin)|<R}
 \xi(\Line_-)
 \Expect{
\min\left\{1,\frac{v(\Line_0)^\alpha}{v(\Line_-)^\alpha}\right\} }
\nu({\d}{\Line_0}) \nu({\d}{\Line_-})
\\
=\,
\left(\tfrac{\gamma-1}{2}\right)^2  
\int_{-\infty}^\infty\int_0^\pi\int_{0}^\infty
\xi(v_-,\theta_-,r_-)
\int_{-R}^R\int_0^\pi\int_{0}^\infty
\left(1\wedge\tfrac{v_0^\alpha}{v_-^\alpha}\right)
v_0^{-\gamma}{\d}{v_0}{\d}{\theta_0}{\d}{r_0}\;
v_-^{-\gamma}{\d}{v_-}{\d}{\theta_-}{\d}{r_-}
\\
=\,
\left(\tfrac{\gamma-1}{2}\right)^2  
\int_{-\infty}^\infty\int_0^\pi\int_{0}^\infty
\xi(v,\theta_-,r_-)
\int_{-R}^R\int_0^\pi\int_{0}^\infty
\left(1\wedge u^\alpha\right)
v (uv)^{-\gamma}{\d}{u}{\d}{\theta_0}{\d}{r_0}\;
v^{-\gamma} {\d}{v}{\d}{\theta_-}{\d}{r_-}
 \\
 \hfill(\text{using }v_-=v, v_0=v u \text{ so that } {\d}{v_0}{\d}{v_-}=v{\d}{v}{\d}{u})
 \\
=\,
\left(\tfrac{\gamma-1}{2}\right)^2
\int_{-\infty}^\infty\int_0^\pi\int_{0}^\infty
\xi(v,\theta_-,r_-) v^{-(2\gamma-1)}
{\d}{v}{\d}{\theta_-}{\d}{r_-}
\times
2\pi R 
\int_{0}^\infty
\min\left\{1,u^\alpha\right\}
u^{-\gamma}{\d}{u} 
 \\
\quad=\quad
\left(\gamma-1\right) \pi R  
\times
\Expect{\sum_{\Line\in\Pi} \xi(\Line) v(\Line)^{-(\gamma-1)}}
\times
\left(
\int_{0}^1
u^{\alpha-\gamma}{\d}{u}  
+
\int_{1}^\infty
u^{-\gamma}{\d}{u}  
\right)
\end{multline*}
(where the last step uses \(\alpha>0\)).
This is finite only if \(\alpha>\gamma-1\).
Moreover in that case the theory of infinite products implies
that the next scattering cannot occur instantaneously,
and that there is a positive chance of the next scattering taking place
at an intersection with a line faster than the current line.
Since such intersections form a locally finite pattern along the 
current line, this implies that the scattering classes do indeed correspond to the lines, and therefore that the scattering process 
is non-trivial.

If \(\alpha\leq0\) then
\[
\sum_{\Line\in\Pi} \xi(\Line) \sum_{\dist(\Line',\origin)<A}
 s(\Line,\Line';\Pi\setminus\{\Line,\Line'\})
 \quad=\quad
 \sum_{\Line\in\Pi} \xi(\Line) \sum_{\dist(\Line',\origin)<A}
 \min\left\{1,
 \frac{v(\Line')^\alpha}{v(\Line)^\alpha}\right\}
 \quad=\quad\infty\,,
\]
because of density everywhere of the set of lines in \(\Pi\) of speed lower than
 any prescribed positive threshold.
 In case \(0<\alpha\leq\gamma-1\), monotonicity considerations
 mean that it suffices to consider 
 the boundary case \(\alpha=\gamma-1\). 
 The relevant quantity is then
 \[
  \sum_{\Line\in\Pi, v(\Line)\in(0,1)} v(\Line)^{\gamma-1}
  \Indicator{\Line\text{ hits a fixed unit interval}}\,.
 \]
This admits the stochastic lower bound \(\lim_{n\to\infty} H_n\),
where:
\begin{equation}\label{eq:lower-bound}
 H_n\quad=\quad
 \sum_{r=1}^n N_r a_r^{\gamma-1}
 \,,
\end{equation}
where the \(N_r\) are independent Poisson\((1)\) random variables
and \(0<a_{r+1}<a_r\leq1\) is chosen so that
\[
 \frac{a_{r+1}^{-\gamma-1}-a_{r}^{-\gamma-1}}{\gamma-1}\quad=\quad1\,.
\]
Decomposing the process \(H_n\) as the sum of
a convergent \(L^2\) martingale and a divergent harmonic sum,
\[
 H_n \quad=\quad 
 \left(\sum_{r=1}^n a_r^{\gamma-1} \times \left(N_r-1\right)\right)
 + \sum_{r=1}^n a_r^{\gamma-1}\,,
\]
we deduce that the lower bound, which is the limit of \eqref{eq:lower-bound} as \(n\to\infty\),
almost surely diverges to \(+\infty\).
\end{proof}

We can now formally define the notion of a \SIRSNRRF.
\begin{defn}\label{def:SIRSN-RRF}
 Consider a Poisson line \SIRSN 
 or \SIRSN candidate \(\Pi\) with parameter \(\gamma\geq2\).
 A \SIRSNRRF based on \(\Pi\) 
 (when quenched by conditioning on \(\Pi\))
is a balanced
delineated reversible scattering process
on state-space
\(\StateSpace\).
with scattering classes given by the lines of \(\Pi\),
and satisfying
similarity-equivariance,
with zero deficit, with \(\kappa(\Line)=v(\Line)^\alpha\)
for some \(\alpha>\gamma-1\).
\end{defn}

In summary, the (scattering) probability for 
the \SIRSNRRF switching from \(\Line_1\) to \(\Line_2\) is given by
\begin{equation}\label{eq:switch}
 s_a \quad=\quad s_{(\Line_1,\Line_2)}
 \quad=\quad \min\left\{1, \frac{v_2^\alpha}{v_1^\alpha}\right\}
 \qquad \text {for some } \alpha\in(\gamma-1,\infty)\,.
\end{equation}
Thus the \RRF always switches to faster lines, but switches to slower lines with probability proportional to a power of the relative speed of the new, slower, line.
(After a successful switch to a new line, the new direction of travel is chosen equiprobably.)

\begin{remark}
 The work of this section
 implies that,
 if
 \(\alpha>\gamma-1\) and there is zero defect,
 then
 the
 equilibrium measure 
 \(\pi_{\Line_1,\Line_2}=\min\{v(\Line_1)^\alpha,v(\Line_2)^\alpha\}\)
 automatically generates scale-invariance
 for the {\SIRSNRRF} 
 viewed as a random process in a random environment. 
 The constraint \(\alpha>\gamma-1\)
 ensures that the scattering times 
 \(0=\tau_0<\tau_1<\tau_2<\ldots< \zeta\)
 are well-defined as the line switching times
 mentioned in the procedure described after
 Definition \ref{def:SIRSN-RRF}.
\end{remark}

\begin{remark}\label{rem:similarity-invariance}
 If \(Z\) is a \SIRSNRRF of parameter \(\alpha\) on \(\Pi\),
 and \({\Similarity}\) is a similarity, then
 \({\Similarity}Z\) is a \SIRSNRRF of parameter \(\alpha\) on \({\Similarity}\Pi\). 
 This follows from similarity-invariance of the scattering probabilities
 and consideration of the simulation algorithm
for such a scattering process \(Z\), given
 in the proof of Theorem \ref{thm:mh-ratio}.
\end{remark}

\begin{lem}\label{lem:irreducible}
 A \SIRSNRRF based on a Poisson line \SIRSN 
 or \SIRSN candidate \(\Pi\)
 (when quenched by conditioning on \(\Pi\))
 forms an irreducible
 Markov chain on the state-space 
 \(\StateSpace\).
\end{lem}
\begin{proof}
 It suffices to show that the \SIRSNRRF can move from \((\Line_-,\Line_0)\)
 to \((\Line_0,\Line_+)\). 
 
 The key observation is that the \SIRSNRRF is always compelled to switch onto a faster line, but may or may not choose to switch to a slower line.
 
 Consequently,
 if these intersections are not separated
 by a line of greater speed than \(v(\Line_0)\) then the
 \SIRSNRRF can travel from
 \((\Line_-,\Line_0)\)
 to \((\Line_0,\Line_+)\)
 in a single move. If there are such lines, then consider the sequence of cells
 from the 
 Crofton tessellation formed by \(\Pi_{\geq v(\Line_0)}\) 
 which intersect the segment
 of \(\Line_0\)
between \((\Line_-,\Line_0)\)
 and \((\Line_0,\Line_+)\). 
 With positive probability, the \SIRSNRRF can move 
 from \((\Line_-,\Line_0)\)
 along \(\Line_0\) in the direction of \((\Line_0,\Line_+)\),
 but has to switch to the Crofton tessellation when it is first encountered.
 The \SIRSNRRF can then use the boundaries of these cells 
 to move 
 to a point on \(\Line_0\) also lying on the
 boundary Crofton cell containing \((\Line_0,\Line_+)\);
 \(\Line_0\) can then be used to move to \((\Line_0,\Line_+)\).
\end{proof}

The continuous-time variant of the {\SIRSNRRF},
\((X_t:t\geq0)\),
can be recovered from the sampled process \((Z_n=X_{\tau_n}:n\geq0)\) 
simply by interpolating between sampling points,
requiring the \RRF \(X\) 
to travel at top permissible speed \(Y_{\tau_n}\) between
scattering times \(\tau_n\) and \(\tau_{n+1)}\).
In principle there is the possibility that the resulting continuous-time process might explode to infinity
in finite time \(\zeta<\infty\). 
We shall discuss this further in
section \ref{sec:conclusion}.

In the next section we address the question of
the long-run behaviour of the
(log-)speed process \(\log(Y)\) of the {\SIRSNRRF}.

 \section[Environment viewed from the SIRSN-RRF]{Environment viewed from the \SIRSNRRF}\label{sec:environment} 
 
We now focus on the 
(discrete-time) \SIRSNRRF \(Z\) of index \(\alpha>\gamma-1\) 
based on the planar Poisson line \SIRSN of parameter \(\gamma>2\),
or even the non-\SIRSN case of \(\gamma=2\), 
as discussed in Section \ref{sec:SIRSN-RRF}. 
This scattering process
can be viewed as possessing a
random and \(\Pi\)-dependent state-space 
\(\StateSpace\).
 Recall that the lines 
 of \(\Pi\)
 are speed-marked, 
 so the state 
 \((\Line_0,\Line_1)\) of \(Z_1\) includes information on the speed \(v(\Line_0)\)
 previous to the switch
 and also the current speed \(v(\Line_1)\).
  Conditioned on \(\Pi\), the process \(Z\)
  is Markovian with a discrete invariant measure
\(\pi_{(\Line_0,\Line_1)}
 =\min\left\{v(\Line_0)^\alpha, v(\Line_1)^\alpha\right\}\)
 for some \(\alpha>\gamma-1\),
 with respect to which \(Z\) satisfies dynamical reversibility.
The discrete invariant measure 
is never summable, since any summation has to extend over 
all the intersection points 
of the stationary line process \(\Pi\).
Consequently 
a stationary version of \(Z\) cannot exist,
and indeed the invariant measure is defined only up to a positive multiplicative constant.
Nonetheless we will see that
the environment viewed from \(Z\) can be converted into a stationary process
(following the classic construction for a random walk in a random environment), 
so long as it is reduced by centering, rotation, \emph{and} (most especially) rescaling.

To begin with, consider the  
\RRF \(Z\) in its quenched environment \(\Pi\).
This dynamically reversible
process can be related to a 
symmetric Dirichlet form which is quenched 
(conditioned on \(\Pi\))
and defined
for the random state-space 
\(\StateSpace\)
as follows:
suppose \(f\) and \(g\) are functions on
\(\StateSpace\) satisfying the symmetry condition
\(f(a)=f(\widetilde{a})\), \(g(b)=g(\widetilde{b})\).
Then
\begin{multline}\label{eq:quenched-dirichlet}
 B^\text{quenched}(f,g) \quad=\quad 
 \sum_{a\in\StateSpace} 
 \pi_a f(a)  \Expect{g({Z}_1) | Z_0=a, \Pi}
\\
\quad=\quad
 \underset{a\neq b\in\StateSpace}{\sum\sum}
 \pi_a p_{a,\widetilde{b}} 
 \;  f(a)g(b )
\\
\quad=\quad 
 \underset{a\neq b\in\StateSpace}{\sum\sum}
 \pi_a\omega_{a,\widetilde{b}} s_{\widetilde{b}} 
 \;  f(a )g(b )
 \quad=\quad
 \underset{a\neq b\in\StateSpace}{\sum\sum}
 \pi_b p_{b,\widetilde{a}}
 \;  f(a )g(b )
 \,,
\end{multline}
where the last step arises from dynamical reversibility
(since \(\pi_a p_{a,\widetilde{b}}=\pi_b p_{b,\widetilde{a}}\)),
and establishes the symmetry of the quenched 
Dirichlet form.
Note that the
equilibrium probabilities \(\pi_a\),
the transition probabilities \(p_{a,\widetilde{b}}\),
and
the 
transmission probabilities \(\omega_{a,\widetilde{b}}\)
all
depend implicitly on the random environment given 
by
the Poisson line \SIRSN 
or \SIRSN candidate \(\Pi\).

A Cauchy-Schwarz argument shows that
the Dirichlet form \(B^\text{quenched}(f,g)\)
is well-defined if the
functions \(f\) and \(g\) belong to the random Hilbert space \(\mathfrak{H}_\Pi\)
of functions \(h\)
defined on \(\StateSpace\) 
satisfying the symmetry requirement 
\(h(\Line_1,\Line_2)=h(\Line_2,\Line_1)\)
and
\begin{equation}\label{eq:quenched-hilbert-space}
 \sum_a h(a)^2 \pi_a
 \quad=\quad
 \underset{a\in\StateSpace}{\sum} 
 h(a)^2
 \pi_{a}
 \quad<\quad\infty
 \,.
\end{equation}

For completeness of exposition, we observe that 
measure-theoretic details for
such symmetric
\(h(a)=h(\Line_0,\Line_1)\)
can be dealt with 
by viewing
\(h(\Line_0,\Line_1)=h(x,\theta_0,v_0,\theta_1,v_1)\)
as a measurable function of one planar and four real variables,
using the \(2:1\) mapping \(\StateSpace\to\Reals^2\)
determined by 
\(\Line_0\cap\Line_1=\{x\}\) 
to deliver
the planar variable \(x\in\Reals^2\), while
\(\theta_i\) signifies the direction,
\(v_i\) the speed of \(\Line_i\).
We can thus
regard \(f\) and \(g\) as functions of \(\Reals^2\times(0,\pi)\times(0,\infty)\times(0,\pi)\times(0,\infty)\).
To be pedantic, we focus on functions in
the subspace of  \(\mathfrak{H}_\Pi\)
which is the \(L^2\) closure 
of  
\(h(x,\theta_0,v_0,\theta_1,v_1)\)
which depend continuously on the arguments
\(x,\theta_0,v_0,\theta_1,v_1\)
and depend only on finitely many evaluations
of events involving whether specified open subsets
of \(\Reals^2\) are hit by \(\Pi\)-lines of speeds exceeding specified positive thresholds.

By Fubini-argument methods
this leads us to consider an annealed Dirichlet form
defined for functions 
\(h(a;\Pi)=h(x,\theta_0,v_0,\theta_1,v_1;\Pi)\)
in the deterministic Hilbert space \(\mathfrak{H}\)
of functions
which also depend on the random pattern specified
by \(\Pi\),
defined by
\begin{equation}\label{eq:annealed-hilbert-space2}
 \mathfrak{H}
 \quad=\quad
 \left\{
 h
 \;:\;
 \Expect{
 \underset{a\in\StateSpace}{\sum}  
 h(a;\Pi)^2 \pi^\Pi_a}<\infty\;\
\text{ with }\;
 h(a;\Pi)=h(\widetilde{a},\Pi)
 \right\}\,.
 \end{equation}
%
Again we restrict to functions in
the \(L^2\) closure 
of  
\(h(x,\theta_0,v_0,\theta_1,v_1)\)
which depend continuously on the arguments
\(x,\theta_0,v_0,\theta_1,v_1\)
and depend only on finitely many evaluations
of events involving whether specified open subsets
of \(\Reals^2\) are hit by \(\Pi\)-lines of speeds exceeding specified positive thresholds.
In particular, measurability of functions
in the subspace will use
the \(\sigma\)-algebra
\(\sigma\{\Pi_{\geq u}:u>0\}\),
where \(\Pi_{\geq u}\) is the locally finite Poisson line process
of \(\Pi\)-lines of speed exceeding \(u\),
viewed as a random closed set and
endowed with the hitting \(\sigma\)-algebra 
(also called \emph{Effros \(\sigma\)-algebra})
generated by hitting events \([\Pi_{\geq u}\cap K\neq\emptyset]\)
for compact sets \(K\subset\Reals^2\)
\cite[\S6.1.2]{ChiuStoyanKendallMecke-2013}.

The annealed Dirichlet form is given by
\begin{multline}\label{eq:annealed-dirichlet}
 B(f,g) \quad=\quad 
 \Expect{
 \sum_{a\in\StateSpace} 
 \pi^\Pi_a f(a;\Pi\setminus a)  
 \Expect{g(\widetilde{Z}_1;\Pi\setminus Z_1) | Z_0=a, \Pi}
 }
 \\
 \quad=\quad
 \Expect{
 \underset{a\neq b\in\StateSpace}{\sum\sum}
 \pi^\Pi_a p_{a,\widetilde{b}}^\Pi 
 \;  f(a;\Pi\setminus a)g(b;\Pi\setminus b)
 }
 \\
\quad=\quad
 \Expect{
 \underset{a\neq b\in\StateSpace}{\sum\sum}
 \pi^\Pi_a\omega_{a,\widetilde{b}}^{\Pi\setminus(a,b)}
 s_{\widetilde{b}} 
 \;  f(a;\Pi\setminus a)g(b;\Pi\setminus b)
 }\,,
\end{multline}
where \(f\) and \(g\) 
are functions of the random
environment \(\Pi\) as well as of 
\(\Reals^2\times(0,\pi)\times(0,\infty)\times(0,\pi)\times(0,\infty)\),
and both belong to \(\mathfrak{H}\).
The 
superscripts in
\(\pi^\Pi_a\),
\(p_{a,\widetilde{b}}^\Pi\) 
and \(\omega_{a,\widetilde{b}}^{\Pi\setminus(a,b)}\),
in \eqref{eq:annealed-dirichlet}
and \eqref{eq:annealed-hilbert-space2}
emphasize dependence
on the environment \(\Pi\) as well as \(a\) and \(b\). 
We use an abbreviated notation 
\(\Pi\setminus a=\Pi\setminus\{\Line_-,\Line_+\}\) 
when \(a=(\Line_-,\Line_+)\),
and 
\(\Pi\setminus (a,b)=\Pi\setminus\{\Line_-,\Line_+,\Line_0,\Line_1\}\) 
when \(a=(\Line_-,\Line_+)\) and \(a=(\Line_0,\Line_1)\).

The annealed 
symmetric
Dirichlet form \eqref{eq:annealed-dirichlet}
can be associated with
the (rather trivial) augmentation of the Markov chain
\(Z\) which is given by
\(((Z_n,\Pi):n\geq0)\).
Thus the augmentation simply
consists of adding the time-constant random process
\(\Pi\). 
Note that we can use \(B(f,g)\) to recover the
joint distribution of \(Z_0\) and \({Z}_1\),
and hence the
conditional probability distribution
\(\Law{{Z}_1 | {Z}_0, \Pi}\).
So knowledge of the annealed 
symmetric
Dirichlet form \eqref{eq:annealed-dirichlet}
identifies the annealed stochastic dynamics 
of the \SIRSNRRF \(Z\).

We now introduce
the notion of 
the
\emph{relative environment process} \(\Zenv\)
and the
\emph{reduced relative environment process} \(\Zenv^{(0)}\)
for \(Z\);
\(\Zenv^{(0)}_n\) is the environment \(\Pi\) viewed
from
\(Z_n=(\Line_{-,n},\Line_{0,n})\),
obtained by
removing the lines \(\Line_{-,n}\) and \(\Line_{0,n}\)
(reduction),
then 
translating, rotating and rescaling 
\(\Pi\setminus\{\Line_{-,n},\Line_{0,n}\}\)
into a standard form
(relativization).
In detail, for each state
\(a=(\Line_{-},\Line_{0})\)
we introduce a proper similarity \(\Similarity_a\)
whose inverse can be used to deliver the required standard form.
(Recall that a
similarity is simply an affine-linear transformation
of Euclidean space:
a \emph{proper} similarity
is one which preserves the sign of the area differential.)
If \(\Line_-\cap\Line_0=\{z\}\) then 
we require that \(\Similarity_a \origin=z\);
furthermore \(\Similarity_a\) must send
the \(x\)-axis \({\Line_{*}^0}\) (with 
sense given by standard direction, and unit speed) 
to the line \(\Line_0\) (with prescribed sense
and rescaling so it has the required speed \(v(\Line_0)\));
finally we require that \(\Similarity_a\) sends
\({\Line_{*}^\theta}\) to \(\Line_-\), where
the line
\({\Line_{*}^\theta}\) passing through \(\origin\)
makes angle \(\theta=\ArcAngle(\Line_-,\Line_0)\) with 
\({\Line_{*}^0}\).
The scaling component of the similarity \(\Similarity_a\) is fixed
by the requirement that \(\Line_*^0\) has unit speed:
as a consequence the speed of \(\Line_*^\theta\)
must be \(v(\Line_-)/v(\Line_0)\). These requirements 
uniquely
define the 
proper similarity \(\Similarity_a\).

\begin{defn}\label{def:reduced-relative-environemt}
 The \emph{relative environment} \(\Zenv_n\) of \(Z_n\) is given by
 \(\Similarity_{Z_n}^{-1}\Pi\).
 The \emph{reduced relative environment} \(\Zenv^{(0)}_n\)
 of
 \(Z_n=(\Line_-,\Line_0)\) is 
 obtained by removing the lines 
 \(\Similarity_{Z_n}^{-1}\Line_-\) and \(\Similarity_{Z_n}^{-1}\Line_0\)
 corresponding to 
 \(\Line_-\) and \(\Line_0\): 
 \[
 \Zenv^{(0)}_n \quad=\quad 
 \Similarity_{Z_n}^{-1}\left(\Pi\setminus Z_n\right)
 \quad=\quad
 \Similarity_{Z_n}^{-1}\left(\Pi\right)\setminus \Similarity^{-1}_{Z_n} Z_n
 \quad=\quad
 \Similarity_{Z_n}^{-1}\left(\Pi\right)
 \setminus\left\{\Similarity_{Z_n}^{-1}\Line_-,\Similarity_{Z_n}^{-1}\Line_0\right\}\,.
 \]
\end{defn} 
So the relative environment of \(Z_n\)
can be parametrized by 
 the relative speed \(v(\Line_-)/v(\Line_0)\)
of the immediately preceding line \(\Line_{-}\)
when compared with the current speed \(\Line_0\),
the angle 
\(\ArcAngle(\Line_-,\Line_0)\)
between current and immediately preceding lines,
and
the point pattern 
\(\Zenv^{(0)}_n\).

The transmission
probability 
\(\omega_{a,\widetilde{b}}^{\Pi\setminus(a,b)}\) 
in \eqref{eq:annealed-dirichlet}
must vanish unless 
\(a\) and \(b\) belong to the same scattering class \(\Equiv\). Moreover in this case the states \(a\) and \(b\)
must share a line:
\(a=(\Line_-,\Line_0)\) and \(b=(\Line_+,\Line_0)\).
Applying dynamical reversibility
of the quenched process,
we find
\(\pi_a/s_{\widetilde a}=v(\Line_0)^\alpha=\kappa(\Equiv)\) is a function of the scattering class \(\Equiv\) alone.
Hence
\eqref{eq:annealed-dirichlet}
can be rewritten as
\begin{multline*}
 B(f,g)\quad=\quad
 \Expect{\sum_{\Equiv\in\EquivClass}
 \left(
 \underset{a\neq b\in\EquivClass}{\sum\sum}
 \kappa(\Equiv)
 s_{\widetilde{a}}
 \omega_{a,\widetilde{b}}^{\Pi\setminus(a,b)}
 s_{\widetilde{b}} 
 \;  f(a;\Pi\setminus a)g(b;\Pi\setminus b)
 \right)
 }
 \\
 \;=\;
  \operatorname{\mathbb{E}}\Bigg[
  \sum_{\Line_0\in\Pi}
 v(\Line_0)^\alpha
 \Bigg(
 \underset{\Line_-\neq \Line_+\in\Pi\setminus\{\Line_0\}}{\sum\sum}
 s_{(\Line_0,\Line_-)}
 f((\Line_-,\Line_0);\Pi\setminus \{\Line_-,\Line_0\})
 \;
 \\
 \qquad
 s_{(\Line_0,\Line_+)} 
 g((\Line_+,\Line_0);\Pi\setminus \{\Line_+,\Line_0\})
 \;
 \omega_{a,\widetilde{b}}^{\Pi\setminus(a,b)}
 \Bigg)
 \Bigg]
 \,,
\end{multline*}
where (compare Theorem \ref{thm:delineated})
\[
 \omega_{a,\widetilde{b}}^{\Pi\setminus(a,b)}
 \quad=\quad
 \omega_{(\Line_-,\Line_0),(\Line_0,\Line_+)}^{\Pi\setminus\{\Line_-,\Line_0,\Line_+\}}
 \quad=\quad
 \mathop{\underset{\Line\text{ separates}}{\prod}}_{
 \Line_0\cap\Line_+
 \text{ and } 
 \Line_0\cap\Line_+}
 \left(1- s_{(\Line_0,\Line)}\right)
 \,.
\]
Note that 
the stochastic dynamics of \(Z\) are invariant
under similarity transformations,
because they depend only on the
scattering probabilities
\(s_{(\Line_0,\Line_-)}\) and \(s_{(\Line_0,\Line_+)}\),
and the transmission probability
\(\omega_{a,\widetilde{b}}^{\Pi\setminus(a,b)}\),
all of which possess this invariance.
It follows that the relative environment process
\(\Zenv=(\Zenv_n=\Similarity_{Z_n}^{-1}\Pi:n\geq0)\),
when quenched by conditioning on \(\Pi\),
is again a Markov chain.

However we can say more.
The 
annealed Dirichlet form \eqref{eq:annealed-dirichlet} 
can be used to 
establish that \(\Zenv\) 
\emph{when not conditioned on \(\Pi\)}
(hence annealed)
forms a stationary process for suitably
distributed random initial
starting points \(X_0\in\StateSpace\), 
and it can then be used to compute
the ensuing stationary distribution of \(\Zenv\).
The theory is closely related to that of Palm
conditioning for point processes,
and similarly requires careful interpretation
although the underlying idea is simple enough:
for some \(c>0\),
the \SIRSN candidate \(\Pi\) is conditioned to have at least one intersection within \(c\) of the origin \(\origin\), such that the current line
has
speed exceeding \(c\),
and the intersection
with the previous line 
is chosen with weight based on the probability of switching. 
Our results will cover
the evolution of the relative environment process \(\Zenv\)
for
\(Z\) 
begun at one of these intersection points 
chosen according to the indicated weighting. 

To facilitate our argument, we first establish
a factorization result
for suitable \(\pi\)-weighted sums over 
\(\StateSpace\).

\begin{lem}\label{lem:pi-weighted-sum}
 Given a Poisson line process 
 \(\Pi\) based on the parameter \(\gamma>1\), consider a non-negative measurable
 function \(\xi((\Line_-,\Line_0);\Pi\setminus\{\Line_-,\Line_0\})\),
 defined for \((\Line_-,\Line_0)\in\StateSpace\),
 which admits a factorization
 \begin{multline}\label{eq:fn-factor}
  \xi((\Line_-,\Line_0);\Pi\setminus\{\Line_-,\Line_0\})\quad=\quad
    \xi_\text{invar}((\Line_-,\Line_0); \Pi\setminus{\Line_-,\Line_0})
  \;\times\;
  \\
  \;\times\;
  \xi_\text{speed}(v(\Line_0))
  \;\times\;
  \xi_\text{config}(\sdist(\Line_-\cap\Line_{0},\Line^{\bot}_{0}\cap\Line_0), 
  \sdist(\Line_0,\origin), \ArcAngle(\Line_0,{\Line_{*}^0})))
  \,,
 \end{multline}
 where \(\Line^\bot_0\) is the line through \(\origin\)
 perpendicular to \(\Line_0\).
 Here \(\xi_\text{invar}\) is a similarity-invariant function of 
 its arguments,
 while \(\xi_\text{speed}\) is a function of the current speed
 and \(\xi_\text{config}\) is a function of three parameters
 describing the location and orientation of the configuration \((\Line_-,\Line_0)\).
 Suppose the intersections 
 \(a=(\Line_-,\Line_0)\in\StateSpace\)
 are weighted by 
 \(\pi_a=\min\{v(\Line_-)^\alpha,(\Line_0)^\alpha\}\)
 for some fixed \(\alpha\in(\gamma-1,\infty)\).
 Then
 \begin{multline}\label{eq:int-factor}
  \Expect{
  \underset{\Line_-\neq\Line_0\in\Pi}{\sum\sum}
   \xi((\Line_-,\Line_0);\Pi\setminus\{\Line_-,\Line_0\})\;
   \pi_{(\Line_-,\Line_0)}
  }
  \quad=\quad
  \\
  \left(\frac{\gamma-1}{2}\right)^2
  \int_{-\infty}^\infty\int_{-\infty}^\infty\int_0^\pi
  \xi_\text{config}(s,r,\theta)
  {\d}\theta {\d}s {\d}r
  \;\times\;
  \int_0^\infty \xi_\text{speed}(v) v^{\alpha-2\gamma+1}{\d}{v}
  \;\times\;
  \\
  \;\times\;
  \int_{-\infty}^\infty\int_0^\pi
  \Expect{\xi_1(e^t,\phi,\Pi)}
  \min\{e^{(\alpha - (\gamma-1))t}, e^{-(\gamma-1)t}\}
  \sin\phi {\d}{\phi}\,{\d}{t}
  \,.
 \end{multline}
 Here \(\xi_1(e^t,\phi,\Pi)\) is defined in terms of 
 the similarity-invariant function \(\xi_\text{invar}\)
 by
 \[
  \xi_1\left(
    \frac{v(\Line_-)}{v(\Line_0)}, 
    \ArcAngle(\Line_-,\Line_0),
    \Similarity_{(\Line_-,\Line_0)}^{-1}(\Pi\setminus\{\Line_-,\Line_0\})\right)
  \quad=\quad
  \xi_\text{invar}((\Line_-,\Line_0); \Pi\setminus\{\Line_-,\Line_0\})\,.
 \]
 \end{lem}
\begin{proof}
 As in the proof of Theorem \ref{thm:similarity-invariance}, first
 apply the Slivnyak-Mecke theorem 
 (Theorem \ref{thm:Slivnyak-Mecke}) twice in succession
 to the left-hand side of \eqref{eq:int-factor}:
 \begin{multline*}
  \Expect{
  \underset{\Line_-\neq\Line_0\in\Pi}{\sum\sum}
   \xi((\Line_-,\Line_0);\Pi\setminus\{\Line_-,\Line_0\})
   \; \pi_{(\Line_-,\Line_0)}
  }
  \quad=\quad
  \\
  \int\int
  \Expect{\xi((\Line_-,\Line_0);\Pi)}
  \pi_{(\Line_-,\Line_0)}
  \nu({\d}{\Line_-})\nu({\d}{\Line_0})\,.
 \end{multline*}
Using the representation corresponding 
to \eqref{eq:plp-sirsn-intensity}
for \(\nu({\d}{\Line_0})\)
(based on \(\origin\) and \({\Line_{*}^0}\)
for reference point and line),
and the representation corresponding 
to \eqref{eq:plp-sirsn-intensity-alt}
for \(\nu({\d}{\Line_-})\)
(based on \(\Line_0\cap{\Line_{*}^0}\) 
and \(\Line_0\) for reference point and line),
we obtain
\begin{multline*}
  \Expect{
  \underset{\Line_-\neq\Line_0\in\Pi}{\sum\sum}
   \xi((\Line_-,\Line_0);\Pi\setminus\{\Line_-,\Line_0\})
   \; \pi_{(\Line_-,\Line_0)}
  }
  \quad=\quad 
  \\
 \left(\frac{\gamma-1}{2}\right)^2 
 \int_{-\infty}^\infty \int_0^\pi \int_0^\infty
 \Bigg(
 \int_{-\infty}^\infty \int_0^\pi \int_0^\infty
 \Expect{\xi_\text{invar}\left(\frac{v_-}{v_0}, \phi, 
 \Similarity_{(\Line_-,\Line_0)}^{-1}(\Pi)\right)}
 \times
 \\
 \times
 \xi_\text{speed}(v_0)
 \times
 \xi_\text{config}(s, r, \theta)
 \times
 \min\{v_-^\alpha,v_0^\alpha\}
 v_-^{-\gamma}\sin\phi {\d}{v_-} {\d}{\theta} {\d}{s}
 \Bigg)
 v_0^{-\gamma}{\d}{v_0} {\d}{\theta} {\d}{r}\,.
\end{multline*}
But scale invariance implies that 
\(\Expect{\xi_\text{invar}\left(\tfrac{v_-}{v_0}, \phi, 
 \Similarity_{(\Line_-,\Line_0)}^{-1}(\Pi)\right)}
 =
 \Expect{\xi_\text{invar}\left(\tfrac{v_-}{v_0}, \phi, 
 \Pi\right)}\).
The result now follows by a simple change of coordinates:
set \(v_0=v\) and \(v_-=v_0e^t\) so that
\({\d}{v_-}{\d}{v_0}=e^t v {\d}{v}{\d}{t}\).
\end{proof}

We can now state and prove the main theorem of this section.
\begin{thm}\label{thm:stationarity}
 Given a \SIRSNRRF \(Z\) on a Poisson line \SIRSN
 or \SIRSN candidate \(\Pi\),
 parametrized by \(\alpha>\gamma-1\) where \(\gamma\) is the 
Poisson line \SIRSN parameter,
the relative environment process \(\Zenv\) for the 
\SIRSNRRF \(Z\) can be made stationary if its initial distribution
can be expressed as three independent components as follows:
\begin{enumerate}[(1)]
 \item\label{result:log-rel-speed}
 the log-relative speed of the line before the current line
 has an asymmetric Laplacian density over \(\Reals\): rate parameter 
 \(\gamma-1\) for positive values, \(\alpha-(\gamma-1)\)
 for negative values;
 \item the angle between current and previous lines has a sine-weighted
 density;
 \item the ensemble 
 \(\Zenv^{(0)}\) of all lines other than current and previous lines
 (the reduced relative environment)
 is distributed as the original speed-marked Poisson line process \(\Pi\).
\end{enumerate}
 \end{thm}
\begin{proof}
Because \(B(f,g)\) is given by
an expression involving the Markovian kernel
\(\pi_a\omega_{a,\widetilde{b}}\),
\[
 B(f,g) 
 \quad=\quad
  \Expect{
 \underset{a\neq b\in\StateSpace}{\sum\sum}
 \pi_a\omega_{a,\widetilde{b}} s_{\widetilde{b}} 
 \;  f(a;\Pi)g(b;\Pi)
 }\,,
\]
it possesses a completion which applies to
the case when \(g\) is bounded
and \(f\)  
satisfies the \(L^1\) condition
\[
  \Expect{
 \underset{a\in\StateSpace}{\sum}  
|f(a;\Pi)| \; \pi_a }<\infty\,.
\]
Suppose \(f\) is  non-negative and admits a factorization
as in Lemma \ref{lem:pi-weighted-sum};
 \begin{multline*}
  f((\Line_-,\Line_0);\Pi)\quad=\quad
    f_\text{invar}((\Line_-,\Line_0); \Pi\setminus{\Line_-,\Line_0})
  \;\times\;
  \\
  \;\times\;
  f_\text{speed}(v(\Line_0))
  \;\times\;
  f_\text{config}(\sdist(\Line_-\cap\Line_{0},\Line_{x,0}\cap\Line_0), 
  \sdist(\Line_0,\origin), \ArcAngle(\Line_0,{\Line_{*}^0})))
  \,,
 \end{multline*}
where \(f_\text{invar}\) is similarity-invariant.
In particular, if \(\phi_-=\ArcAngle(\Line_-,{\Line_0})\)
and
\(v(\Line_-)=e^{t_-} v(\Line_0)\)
then we write
\begin{multline*}
 f_\text{invar}((\Line_-,\Line_0); \Pi\setminus\{\Line_-,\Line_0\})
 \;=\;
 f_\text{invar}\left(\frac{v(\Line_-)}{v(\Line_0)}, 
 \ArcAngle(\Line_-,{\Line_0}),
 \Similarity_{(\Line_-,\Line_0)}^{-1}(\Pi\setminus\{\Line_-,\Line_0\})\right)
 \\
 \quad=\quad
 f_1(e^{t_-}, \phi_-, \Zenv^{(0)}_0)
 \,.
\end{multline*}
Suppose further that the bounded \(g\) is 
itself similarity-invariant: we write
\[
 g(\widetilde{b};\Pi)\quad=\quad
  g((\Line_+,\Line_0);\Pi)\quad=\quad
 g_\text{invar}((\Line_+,\Line_0); \Pi\setminus\{\Line_0,\Line_+\})
 \,.
\]
Since the dynamics of \(Z\) are similarity-invariant, this means that
\[
 \Expect{g(\widetilde{Z}_1;\Pi) | Z_0=a, \Pi}
 \quad=\quad
 \Expect{g(\Similarity\widetilde{Z}_1;\Similarity\Pi) | \Similarity Z_0=a, \Similarity \Pi}
\]
for any similarity \(\Similarity\), so \(\Expect{g(\widetilde{Z}_1;\Pi) | Z_0=a, \Pi}\)
is similarity-invariant
as a function of \(a\) and \(\Pi\setminus a\).
Consequently we may apply Lemma \ref{lem:pi-weighted-sum} to
\begin{multline*}
\xi((\Line_-,\Line_0);\Pi)
\quad=\quad
 f((\Line_-,\Line_0);\Pi)\Expect{g(\widetilde{Z}_1;\Pi) | Z_0=(\Line_-,\Line_0), \Pi}
 \quad=\quad
 \\
 f_\text{invar}((\Line_-,\Line_0); \Pi\setminus\{\Line_0,\Line_-\})\times
 \Expect{
 g_\text{invar}(\widetilde{Z}_1; \Pi\setminus\{\Line_0,\Line_+\})
 \;|\;
 Z_0=(\Line_-,\Line_0), \Pi
 }\times
 \\
 \times
 f_\text{speed}(v(\Line_0))
  \;\times\;
  f_\text{config}(\sdist(\Line_-\cap\Line_{0},\Line_{x,0}\cap\Line_0), 
  \sdist(\Line_0,\origin), \ArcAngle(\Line_0,{\Line_{*}^0})))
  \,.
\end{multline*}
We deduce
\begin{multline*}
B(f,g)\quad=\quad
\Expect{\underset{\Line_-\neq\Line_0\in\Pi}{\sum\sum}
 \pi_{(\Line_-,\Line_0)} f((\Line_-,\Line_0);\Pi)
 \Expect{g(\widetilde{Z}_1;\Pi) | Z_0=(\Line_-,\Line_0), \Pi}}
 \\
 \quad=\quad
  \left(\frac{\gamma-1}{2}\right)^2
  \int_{-\infty}^\infty\int_{-\infty}^\infty\int_0^\pi
  f_\text{config}(s,r,\theta)
  {\d}\theta {\d}s {\d}r
  \;\times\;
  \int_0^\infty f_\text{speed}(v) v^{\alpha-2\gamma+1}{\d}{v}
  \;\times\;
  \\
  \int_{-\infty}^\infty\int_0^\pi
  \min\{e^{(\alpha - (\gamma-1)) t_-},e^{-(\gamma-1)t_-}\}
  \times
  \\
  \times
  \Expect{
  f_\text{invar}((\Line_*^{\phi_-}(e^{t_-}),\Line_*^0);\Pi)
  \Expect{g_\text{invar}(\widetilde{Z}_1;\Pi) | Z_0=(\Line_*^{\phi_-}(e^{t_-}),\Line_*^0), \Pi}
  }
    \sin\phi_-{\d}{\phi_-}\,{\d}{t_-}
\,,
\end{multline*}
where the line \(\Line_*^{\phi_-}(e^{t_-})\) meets 
the unit-speed (\(x\)-axis) \(\Line_*^0\)
at \(\origin\), making an angle \(\phi_-\), 
and has speed \(e^{t_-}\).

We may deduce the following by choosing 
\(f_\text{config}(s,r,\theta)=\Indicator{s^2+r^2\leq c^2}\) 
and \(f_\text{speed}(v)=\Indicator{v>c}\) for fixed \(c>0\).
Sample the speed-marked line process \(\Pi\),
and sample \(\Line_0\) uniformly at random from the set of lines of \(\Pi\) lying within \(c\) of \(\origin\)
and with speed exceeding \(c\).
Then sample \(\Line_-\) from the lines of \(\Pi\) intersecting \(\Line_0\) and such that (i) the intersection point is within 
\(c\) of \(\origin\), 
using sampling weights \(\min\{1,(v(\Line_1)/v(\Line_0))^\alpha\}\)
If there is no such line \(\Line_0\),
or there turn out to be no such intersections, then re-sample \(\Pi\) and repeat till successful.
Use the resulting \((\Line_-,\Line_0)\) as the initial point of the \SIRSNRRF \(Z\).
Then the resulting relative environment process is associated with the following 
Dirichlet form:
\begin{multline*}
 B^\text{relative}(f_\text{invar},g_\text{invar})
 \;=\;
 \Expect{
  \min\left\{1, \left(\frac{v(\Line_-)}{v(\Line_0)}\right)^{\alpha}
  \right\}
  f_\text{invar}((\Line_-,\Line_0);\Pi)
  \Expect{g_\text{invar}(\widetilde{Z}_1;\Pi) | Z_0=(\Line_-,\Line_0), \Pi}
  }
  \\
  \;=\;
  \int_{-\infty}^\infty  \int_0^\pi
  \operatorname{\mathbb{E}}
  \Bigg[f_1(e^{t_-},\phi_-,\Pi) 
  \sum_{\Line_+\in\Pi}
  \omega^{\Pi\setminus\{\Line_+\}}_{\origin, (\Line_*^0\cap\Line_+)}
  \min\{1,v(\Line_+)^\alpha\}
  \times
  \\
  \times
  g_1(v(\Line_+),\ArcAngle(\Line_*^0,\Line_+),\Pi) 
  \Bigg]
  \min\{e^{(\alpha - (\gamma-1))t_-}, e^{-(\gamma-1)t_-}\}
  \sin\phi_- {\d}{\phi_-}
  \,
  {\d}{t_-}
  \,.
\end{multline*}
Here
\[
 \omega^{\Pi\setminus\{\Line_+\}}_{\origin, (\Line_*^0\cap\Line_+)}
 \quad=\quad
  \mathop{\underset{\Line\in\Pi\setminus\{\Line_+\}\text{ separating}}{\prod}}_{\origin \text{ and } (\Line_*^0\cap\Line_+)}
\Big(1-s_{\Line_0^*,\Line}\Big)\,.
\]

One further application of the Slivnyak-Mecke theorem
(Theorem \ref{thm:Slivnyak-Mecke}),
using the representation 
\eqref{eq:plp-sirsn-intensity-alt}
for \(\nu({\d}{\Line_+})\),
now yields
\begin{multline}
 B^\text{relative}(f_\text{invar},g_\text{invar})
 \quad=\quad
 \\
 \int_{-\infty}^\infty
 \int_{-\infty}^\infty
   \int_0^\pi
   \;
   \int_{-\infty}^\infty  \int_0^\pi
  \operatorname{\mathbb{E}}
  \Bigg[f_1(e^{t_-},\phi_-,\Pi) 
  \omega^{\Pi}_{[\origin, s]}
  \min\{1,e^{\alpha t_+}\}
  g_1(e^{t_+},\phi_+,\Pi) 
  \Bigg]
  \times
  \\
  \times
  \min\{e^{(\alpha - (\gamma-1))t_-}, e^{-(\gamma-1)t_-}\}
  \sin\phi_- {\d}{\phi_-}
  \,
  {\d}{t_-}
  \;
  e^{-(\gamma-1) t_+}\sin\phi_+ {\d}{\phi_+}
  \,
  {\d}{t_+}
  {\d}{s_+}
  \\
  \quad=\quad
 \int_{-\infty}^\infty
 \int_{-\infty}^\infty
   \int_0^\pi
   \;
   \int_{-\infty}^\infty  \int_0^\pi
\Expect{f_1(e^{t_-},\phi_-,\Pi) g_1(e^{t_+},\phi_+,\Pi)
\;
\omega^\Pi_{[\origin,s}}
  \times
  \\
  \times
  \min\{e^{(\alpha - (\gamma-1))t_-}, e^{-(\gamma-1)t_-}\}
  \min\{e^{(\alpha - (\gamma-1))t_+}, e^{-(\gamma-1)t_+}\}
  \times
  \\
  \times
  \sin\phi_- {\d}{\phi_-}
  \sin\phi_+ {\d}{\phi_+}
  \,
  {\d}{t_-}
  \,
  {\d}{t_+}
  {\d}{s_+}
  \,,
\end{multline}
where \([\origin,s]\) is short-hand for the interval 
along \(\Line_*^0\) with one endpoint given by \(\origin\)
and with the signed length \(s\).

The invariance of \(f_\text{invar}\), \(g_\text{invar}\), and \(f_\text{invar}\) collectively imply that the Dirichlet form \(B^\text{relative}(f_\text{invar},g_\text{invar})\) is symmetric.
This in turn implies that
the relative environment
process
\(\Zenv\) is stationary,
with an invariant measure which is a probability 
measure
which makes
the three coordinates of log relative speed \(t\)
of previous line, angle \(\phi\) between previous and current lines,
and reduced environment \(\Zenv^{(0)}\) independent with joint distribution as follows
\begin{enumerate}
 \item corresponding to the log relative speed \(t\),
 a (possibly asymmetric) Laplacian density
 over \(\Reals\) 
 \citep[ch.~24]{JohnsonKotzBalakrishnan-1995},
 given by 
  \begin{equation}\label{eq:equilibrium-density-reversed}
   f_{\alpha, \gamma}(y)\quad=\quad 
   \begin{cases}
  \frac{(\gamma-1)(\alpha-(\gamma-1))}{\alpha} 
  e^{-(\gamma-1)|y|}              & \text{ when } y\geq 0\,,
    \\
  \frac{(\gamma-1)(\alpha-(\gamma-1))}{\alpha} 
     e^{-(\alpha-(\gamma-1))|y|} & \text{ when } y< 0\,;
   \end{cases}
  \end{equation}

%
%

 \item corresponding to \(\phi\),
 a half-sine density over \([0,\pi)\), given by \(\tfrac12\sin\phi\);
 \item corresponding to the reduced relative
 environment \(\Zenv^{(0)}\),
 a distribution which agrees with that of the
 underlying
 \SIRSN
 candidate \(\Pi\).
\end{enumerate}
This completes the proof.
\end{proof}

We are actually interested in the log-relative speed of the current line with respect to the previous line. 
The distribution of this in stationary state
is readily computed directly from Theorem \ref{thm:stationarity},
bearing in mind that if \(Z\) is at \((\Line_-, \Line_0)\)
then its next state is \((\Line_0,\Line_+)\), where 
\(\Line_+\) is drawn from the 
reduced relative environment \(\Zenv^{(0)}\),
such that \((\Line_0,\Line_+)\) is the first intersection
along \(\Line_0\) is the chosen direction
which is accepted by the rule summarized by the acceptance
probability \eqref{eq:switch}.
We obtain
\begin{cor}\label{cor:log-rel-speed}
 In the situation of Theorem \ref{thm:stationarity},
 when \(\Zenv\) is stationary, 
 the distribution 
 of the log of the speed of the current line
 relative to the speed of the previous line has density
 given by
 the (possibly asymmetric) Laplacian density
 prescribed by \eqref{eq:equilibrium-density-reversed},
  with mean value given by
   \begin{equation}\label{eq:equilibrium-mean}
  \int_{-\infty}^\infty y f_{\alpha, \gamma}(y) {\d}{y}
  \quad=\quad
   \frac{1}{\gamma-1} - \frac{1}{\alpha-(\gamma-1)} 
  \quad=\quad
  \frac{\alpha - 2(\gamma-1)}{(\gamma-1)(\alpha-(\gamma-1))}\,.
 \end{equation}
\end{cor}
\begin{proof}
 Leaving \((\Line_-, \Line_0)\),
 \(Z\) encounters intersections with lines of \(\Zenv^{(0)}\)
 according to a speed-marked Poisson 
 process with intensity measure \(v^{-\gamma}{\d}v{\d}s\),
 where \(v\) is now the speed of the new line relative to 
 the speed of \(\Line_0\) and \(s\) is the scaled distance
 along \(\Line_0\) from \((\Line_-, \Line_0)\).
 Note that the mark measure does not have finite mass.
 We have to thin this marked Poisson process
 with retention probability \(\min\{1,v^\alpha\}\),
 according to the acceptance probability \eqref{eq:switch}
 (bearing in mind that \(v\) is the relative speed of the new line),
 so the mark distribution of the retained lines does
 have finite mass.
 Accordingly the mark distribution of retained lines,
 and thus the relative speed of the new line, is proportional
 to (hence equal to) \eqref{eq:equilibrium-density-reversed}.
\end{proof}

%
%
%

In particular the log-relative-speed density has zero mean
in the symmetric case,
when \(\alpha = 2(\gamma-1)\), 
in which case the log-relative-speed stationary distribution 
is is symmetric and is
given by the Laplace or double-headed exponential distribution, 
with rate parameter \(\gamma-1\), and density
  \begin{equation}\label{eq:critical-equilibrium-density}
   f_{2(\gamma-1), \gamma}(y)\quad=\quad 
  \frac{\gamma-1}{2} 
  e^{-(\gamma-1)|y|} \,.
\end{equation}

Finally note that the reduced relative environment process, and hence the relative environment process, is very far from being irreducible. 
Indeed, any particular realization of the state
\(\Zenv\) of the reduced environment process
defines a countable set of 
intersection angles \(A=\{\ArcAngle(\Line_1, \Line_2)):
(\Line_1,\Line_2)\in(\Zenv\times\Zenv)\setminus\Delta\}\)
\emph{which remains time-constant under the evolution of 
\(\Zenv\)}.
But \(\Zenv\) has equilibrium distribution given by the \SIRSN candidate \(\Pi\), and
there will be zero probability of \emph{any}
of the countably many intersections
of lines from \(\Pi\) having an angle belonging to a fixed countable set.
Indeed, by the Slivnyak-Mecke theorem
(Theorem \ref{thm:Slivnyak-Mecke}) we know
that for any fixed angle \(\phi\) we must have
\[
 \Expect{\sum_{(\Line_1,\Line_2)\in\StateSpace}
 \Indicator{\ArcAngle(\Line_1,\Line_2)=\phi}
 }
 \quad=\quad
 \int\int\Indicator{\ArcAngle(\Line_1,\Line_2)=\phi}
 \nu({\d}{\Line_1})\nu({\d}{\Line_2})
 \quad=\quad0\,.
\]
Thus conventional Markovian arguments cannot be applied. However, as we will see in the next section, ergodic 
theory allows us to prove the results we need.

 \section[Long-term behaviour of SIRSN-RRF speed]{Long-term behaviour of \SIRSNRRF speed}\label{sec:ergodic}

If \(Z=(Z_0,Z_1, \ldots)\) is the (discrete-time)
\SIRSNRRF 
and if \(V=(V_0,V_1, \ldots)\) yields the 
corresponding sequence of
speeds for the current line,
then the relative speed change \(V_n/V_{n-1}\) can be 
determined using only \(\Zenv_n\) 
the relative
environment (relativized
by centering, rotating, and scaling).
Theorem \ref{thm:stationarity}
implies that the log-relative speed-changes
\(U_n=\log(V_n/V_{n-1})\) form a stationary sequence, if the
initial relative environment is given the stationary distribution
discussed in the previous section,
and thus \(U_0\)
is given the equilibrium density specified in 
Equation \eqref{eq:equilibrium-density-reversed}.
We may therefore apply 
the non-ergodic part of
Birkhoff's ergodic theorem 
(see for example \citealp[Theorem 10.6]{Kallenberg-2010}) to show
\begin{equation}\label{eq:long-run-equation}
 \frac{1}{n}\log(V_n/V_0) \quad=\quad
 \frac{1}{n}\sum_{m=1}^n U_m \quad\to\quad
 \Expect{U_0 | \mathcal{I}}\,,
\end{equation}
where \(\mathcal{I}\) is the \(\sigma\)-algebra of shift-invariant events for the random process \(U\),
and convergence is both almost sure and in \(L^1\).

It immediately follows that,
away from the critical case \(\alpha=2(\gamma-1)\),
the speed \(V_n\) has at least a positive chance of 
either diverging to \(+\infty\) exponentially fast as 
\(n\to\infty\)
(if \(\alpha > 2(\gamma-1)\))
or converging to \(0\) exponentially fast as 
\(n\to\infty\) 
(if \(\alpha < 2(\gamma-1)\)).
We may therefore rule out non-critical cases (\(\alpha\neq2(\gamma-1)\))
in our search for an example which is speed-neighbourhood-recurrent.

Suppose it can be shown that  
\(\Zenv\) (and therefore \(U\)) is ergodic, so that we can replace
\(\Expect{U_0 | \mathcal{I}}\) by \(\Expect{U_0}\) in \eqref{eq:long-run-equation}.
In the non-critical cases discussed above, this means 
we can replace ``has at least a positive chance'' by ``will almost surely end up``.
But in the critical case \(\Expect{U_0}=0\)
is not sufficient in itself
to guarantee neighborhood-recurrence for \(U\).
However 
in the ergodic case
neighborhood-recurrence actually follows rather simply from the celebrated
(but unpublished)
Kesten, Spitzer and Whitman 
range theorem (described by \citealp[page 38]{Spitzer-1964}).
In its original form the range theorem implies concerns recurrence for integer-valued 
stationary ergodic processes of zero mean. The real-valued / neighbourhood
recurrent case is a simple variation on the original idea:
 \begin{thm}{(Kesten-Spitzer-Whitman, real-valued case.)}
 \label{thm:KestenSpitzerWhitman}
  Suppose \(U_1\), \dots, \(U_n\), \ldots form a stationary ergodic sequence,
  with \(\Expect{U_1}=0\).
  Set \(W_n=U_1+\ldots+U_n\).
  Then for all \(\eps>0\) it is the case that
  \[
   \Prob{|W_n-W_0| \leq\eps \text{ infinitely often in }n}\quad=\quad1\,.
  \]
 \end{thm}
 \begin{proof}
  Birkhoff's ergodic theorem guarantees that \(W_n/n\to0\) 
  almost surely, hence 
  \begin{equation}\label{eq:zero-average}
  n^{-1}\sup\{|W_1|,\ldots,|W_n|\}\quad\to\quad0\qquad\text{ almost surely.}
  \end{equation}
  Set \(A_n=[|W_m-W_n|>\eps\text{ for all }m> n]\). 

  Applying Birkhoff's ergodic theorem again,
  \begin{equation}\label{eq:b-e-t}
  n^{-1}(\Indicator{A_1}+\ldots+\Indicator{A_n})\quad\to\quad \Prob{A_0}\,.
  \end{equation}

  If \(m\neq n\) then
  \(|W_m-W_n|>\eps\) on \(A_n\cap A_m\), 
  and therefore a simple packing argument
  shows that 
  \begin{equation}\label{eq:packing}
  n^{-1}(\Indicator{A_1}+\ldots+\Indicator{A_n})
  \quad\leq\quad
  (n\eps)^{-1}\sup\{|W_1|,\ldots,|W_n|\}\,.
  \end{equation}
  Letting \(n\to\infty\) in \eqref{eq:packing} and using \eqref{eq:zero-average},
  we can deduce from \eqref{eq:b-e-t} that \(\Prob{A_0}=0\).
  Consequently 
  \[
\Prob{|W_n-W_0| \leq\eps \text{ at least once for }n>0}\quad=\quad1\,.
  \]
The same argument applies for the sub-sampled process \((W_0, W_m, W_{2m}, \ldots)\),
for any sub-sampling gap \(m>0\). Therefore the event
\(\bigcap_m \bigcup_{n\geq m} [|W_n-W_0|\leq\eps]\) happens
almost surely.
Consequently it is almost sure that \(W_n\) returns to within \(\eps\) of
\(W_0\) for infinitely many \(n\) and the result follows.
 \end{proof}

 Accordingly speed-neighborhood-recurrence is established in the critical
 case \(\alpha=2(\gamma-1)\) if we can show that the reduced environment
 process \(\Zenv\) is ergodic.
 This is the main result of this paper:
\begin{thm}\label{thm:ergodic-l-r-s}
 Given a \SIRSNRRF \(Z\) on a Poisson line \SIRSN
 or \SIRSN candidate \(\Pi\),
 parametrized by \(\alpha>\gamma-1\) where \(\gamma\) is the 
Poisson line \SIRSN parameter,
the relative  environment process \(\Zenv\) for the 
\SIRSNRRF \(Z\)  is ergodic.
\end{thm}
 
\begin{proof}
The key part of the argument is a
variation on an argument of \citet[Lemma 1, page 82]{Kozlov-1985}.

Firstly, consider the process \(\Zenv\) of the
relative 
environment viewed from the particle. 
Let \(h\) be a bounded harmonic function on the
relative 
environment state-space (harmonic with respect to the
process \(\Zenv\)). 
It suffices to show that
\(h(\Zenv_n)\) is almost surely constant in (discrete) time \(n\). 

Now consider
\[
 \Expect{(h(\Zenv_0)-h(\Zenv_1))^2} \quad=\quad 2 \,\Expect{h(\Zenv_0)^2} - 2 \,\Expect{h(\Zenv_0)h(\Zenv_1)}
\quad=\quad 0\,,
\]
where the first step follows from stationarity and the second because \(h(\Zenv)\) is a martingale. 
Consequently \(\Prob{h(\Zenv_1) = h(\Zenv_0)}=1\).

Secondly, using \(\Zenv\) to explore the network represented by \(\Pi\), 
we see that there is a \(\Pi\)-measurable random variable \(H = H(\Pi)\) such that \(h(\Zenv_n) = H(\Pi)\)
for all \(n\), 
for environment \(\Pi\).
Moreover \(H(\Pi)\) inherits similarity-invariance from \(\Zenv\).
It follows 
from the ergodicity of \(\Pi\) (Theorem \ref{thm:ergodic})
that  \(H(\Pi)\) must be non-random.
\end{proof}

This together with Theorem \ref{thm:KestenSpitzerWhitman} implies speed-neighborhood-recurrence for the \RRF \(Z\), as required.
It also shows that in non-critical cases 
the speed will either almost surely diverge to infinity 
or almost surely converge to zero.
Accordingly a version of Conjecture \ref{conj:complete-stop}
holds for the randomly-broken local \(\Pi\)-geodesics formed by a critical
\SIRSNRRF:
in the {critical case} \(\alpha=2(\gamma-1)\) 
the \RRF provides a ``randomly-broken 
local \(\Pi\)-geodesic'', which avoids slowing down to zero speed (or speeding up to infinite speed).

We conclude this section with a formal statement of the 
speed-neighborhood-recurrence result.
\begin{thm}\label{thm:nbd-recurrence}
 Let \(\Pi\) be a Poisson line \SIRSN or \SIRSN candidate 
 with parameter \(\gamma\geq2\). Then there exists 
 a (discrete-time) \SIRSNRRF \(Z\)
 on \(\Pi\) (an \RRF with similarity-invariant
 dynamics with zero defect) such that the speed process
 \(V\) (given by \(V_n=v(\Line_0)\) when \(Z_n=(\Line_-,\Line_0)\))
 almost surely returns infinitely often to any neighbourhood
 of \(V_0\).
\end{thm}
\begin{proof}
 Bearing in mind the characterization 
 (Theorem \ref{thm:similarity-invariance})
 of such \SIRSNRRF by index \(\alpha>\gamma-1\), 
 we choose the \SIRSNRRF with critical index \(\alpha=2(\gamma-1)\).
 Considering the relative environment process \(\Zenv\)
 run in stationarity, and noting that 
 in stationarity the mean log-relative 
 speed \(U=\log(V)\) has mean zero
 (Theorem \ref{thm:stationarity}),
 and forms an ergodic process
 (Theorem \ref{thm:ergodic-l-r-s}),
 the desired speed-neighborhood-recurrence result
 follows from the adapted Kesten-Spitzer-Whitman 
 Theorem \ref{thm:KestenSpitzerWhitman}.
\end{proof}

 \section{Conclusion}\label{sec:conclusion}
 
This paper 
defines and characterises \SIRSNRRF (similarity-equivariant 
discrete-time
Rayleigh random flights taking place on
scale invariant random spatial networks)
using
 an axiomatic approach to scattering processes.
It is
shown that the relative environment viewed from the
\SIRSNRRF is ergodic stationary,
and that there exists a critical \SIRSNRRF
whose speed process is neighborhood-recurrent.
We offer this as evidence 
in favour of Conjecture \ref{conj:complete-stop},
that \(\Pi\)-geodesics in a Poisson line \SIRSN never come to a complete
halt, and therefore can be constructed using doubly infinite
sequences of segments taken from the Poisson line \SIRSN.

We note that the abstract approach to scattering set out 
in Section \ref{sec:ASP} merits 
further exploration in its own right.

In conclusion, we briefly discuss some points going beyond
the question of speed-neighbourhood recurrence in the critical case.

A little more can be said concerning the two non-critical cases.
Bearing in mind Corollary \ref{cor:log-rel-speed},
if \(\alpha<2(\gamma-1)\), so that the log-relative-speed distribution 
of the next line relative to the current line
for the \SIRSNRRF has negative mean,
then ergodicity of the relative environment
means that
the \SIRSNRRF process itself must almost surely converge to a limiting random point as time tends to infinity. 
This is because its speed must tend to zero, and so
almost surely it must eventually get trapped in a cell of the proper Poisson line tessellation
formed by \(\Pi_{\geq\eps}\) (recall \(\Pi_{\geq\eps}\)
is the part of \(\Pi\) for which speeds are higher than some \(\eps>0\)).
The trapping occurs as \(\eps\to0\),
since \(\Pi_{\geq \eps}\) is a proper Poisson line tessellation, increasing monotonically as a random set
as \(\eps\to0\), with intensity tending to infinity
and with cells shrinking down to zero size.
Since the \SIRSNRRF has positive chance of not escaping
from
\(\Pi\setminus\Pi_{\geq\eps}\) 
onto \(\Pi_{\geq\eps}\)
for large time,
and has a positive chance of moving freely 
within the current connected component of 
\(\Pi\setminus\Pi_{\geq\eps}\)
(by considerations akin to those of the irreducibility
Corollary \ref{lem:irreducible}),
it follows that the limiting point of the \SIRSNRRF must indeed be random. We call this case the \emph{converging case}.

On the other hand, if \(\alpha>2(\gamma-1)\), then the log-relative-speed distribution of the next line 
relative to the current line for the \SIRSNRRF has positive mean,
and so the \SIRSNRRF process must almost surely diverge to infinity. 
This is because its speed must tend to infinity
(using again ergodicity of the relative environment), and so 
almost surely the process must get trapped on \(\Pi_{\geq v}\) for any \(v>0\).
The divergence occurs as \(v\to\infty\),
since \(\Pi_{\geq v}\) is a proper Poisson line tessellation, 
decreasing monotonically as a random set as \(v\to\infty\), with intensity tending to zero,
and therefore with cells blowing up to arbitrarily large 
size with cell boundaries almost surely being eventually
arbitrarily far from the origin \(\origin\).
 We call this case the \emph{diverging case}.

Whether the case is critical, diverging, or converging,
the discrete-time process is defined for all time
(since it will take an infinite number of jumps for the speed
to exceed all bounds, or to reduce to zero). 
Consequently the continuous-time
variant (defined by interpolating between scatterings using top-speed linear motion)
is defined for all time in the critical case \(\alpha=2(\gamma-1)\).
More generally,
under stationarity
consider 
Palm-conditioning on the current line at 
the \(n^\text{th}\) scattering instant.
The marginal distribution of the distance \(D_n\) travelled till next
scattering must be exponentially distributed,
because the 
pattern of speed-marked intersections on the current line is Poisson. 
Note that \(T_n=D_n/V_n\) is the time
till the next scattering.
Now
\(D_n/V_n^{\gamma-1}\)
is a function of the relative environment \(\Psi_n\)
(using the scaling transformation \eqref{eq:SIRSN-scaling}) and therefore forms an ergodic sequence.
Considering independent thinning of all lines
for which scattering fails,
one calculates the (conditioned) exponential rate 
of \(D_n/V_n^{\gamma-1}=T_n/V_n^{\gamma-2}\)
to be
\(\tfrac{\alpha}{\alpha-(\gamma-1)}\).
By the ergodic theorem it follows that
\(\tfrac{1}{n}\sum_{r=0}^n T_r/V_r^{\gamma-2}\)
converges almost surely to 
\(\tfrac{\alpha-(\gamma-1)}{\alpha}\).
Thus in the non-\SIRSN case of \(\gamma=2\)
it follows that scattering happens at a constant rate
in time, and thus the continuous process will be defined
for all time.
%
%

In the \SIRSN case of \(\gamma>2\), if the diverging case
holds (so \(\alpha>2(\gamma-1)\))
then \(V_n\) will eventually tend to \(\infty\)
at a geometric rate. Thus in that case
almost sure convergence of 
\(\tfrac{1}{n}\sum_{r=0}^n T_r/V_r^{\gamma-2}\)
to a positive constant
forces us to conclude that \(\sum_{r=0}^n T_r\)
diverges to \(\infty\), and therefore again
the continuous process will be defined
for all time.

In contrast, in
the converging case \(\alpha<2(\gamma-1)\)
a similar argument shows that 
the continuous-time process will reach zero-speed in finite time,
trading off the asymptotically linear decrease of the log-speed against
the asymptotically linear increase of
\(\sum_{r=0}^n T_r/V_r^{\gamma-2}\).

In the diverging case
\(\alpha>2(\gamma-1)\)
it is natural to ask whether 
the
(discrete or continuous time)
scattering process achieves a limiting direction as viewed from 
\(\origin\).
In fact it will not do so. This follows by an argument involving:
\begin{itemize}[(i)]
 \item an ergodic theorem for \(\Pi\) under scaling symmetries:
(this is proved in the same manner as Theorem \ref{thm:ergodic} 
but using the \(r\)-\(\theta\) coordinatization
used in Equation \eqref{eq:plp-sirsn-intensity} for the intensity measure \(\nu\) of \(\Pi\),
instead of the 
\(s\)-\(\phi\) parametrization used for Equation
\eqref{eq:plp-sirsn-intensity-alt});
\item a demonstration that if \(C_v\) 
is the zero-cell for \(\Pi_{\geq v}\)
(the \emph{Crofton cell} containing 
the origin for the corresponding tessellation)
then there is \(p>0\) such that if \(Z_0\) lies in \(C_v\)
then \(p\) is a lower bound for
the probability that Z makes a complete circuit of
\(\partial C_v\) when first hitting \(\partial C_v\).
(This is established by noting that by scaling it suffices to consider \(v=1\);
then the relevant probability
can be estimated by thinning \(\Pi_{< 1}\) such that
lines are only retained if they hit \(\partial C_1\)
and additionally will not scatter \(Z\) on the two occasions
when its circuit encounters the line in question.)
In fact  
\cite{Calka-2002} gives
distributional bounds on the out-radius of \(C_1\), though here we need only use the fact that \(C_1\) is 
stochastically bounded;
\item finally combining these two to show
\(\Prob{Z \text{ makes a circuit of } C_n \text{ for infinitely many }n}=1\). 
\end{itemize}
Since \(C_n\) will intersect any fixed line
for large enough \(n\), it follows that \(Z\) will eventually visit any fixed line, and therefore cannot be eventually 
confined within any wedge, and therefore cannot possess a limiting 
direction.

We conclude with some questions for further work.

In the critical case \(\alpha=2(\gamma-1)\)
it is an open question whether the
(discrete or continuous time)
process 
(as opposed to its speed)
is positive-recurrent on neighborhoods of the origin
\(\origin\). 
Note that simple arguments imply that
positive-recurrence on neighbourhoods would
force the conclusion that the process was \emph{point}-recurrent;
if \(Z\)
will always eventually return to a neighbourhood \(A\) of the origin 
then it may
(and therefore eventually will) 
move on to the intersection \(A\cap\Pi_{\geq v}\) between 
the neighbourhood and
the proper Poisson line process \(\Pi_{\geq v}\) (choosing the
positive speed \(v\) depending on \(\Pi\)
so that \(A\cap\Pi_{\geq v}\neq\emptyset\)); irreducibility (Lemma \ref{lem:irreducible}) then implies 
that \(Z\) has 
positive chance of visiting 
any specified point on 
\(A\cap\Pi_{\geq v}\), and 
therefore will succeed in doing so eventually 
after repeated
visits to \(A\cap\Pi_{\geq v}\).
However there is some evidence that in fact
\(Z\) is transient:
\(Z\) is caricatured by the two-dimensional integrated Brownian motion
\(\int V{\d}s\) (for \(V\) a \(2\)-dimensional Brownian motion),
which can be shown almost surely to have only finite total length of
path within any bounded neighbourhood of its starting point,
and thus to be transient in the sense of almost surely 
eventually leaving this starting point never to return.

It is natural to ask whether some kind of central limit behaviour
can be established. 
Certainly this question makes sense for the log-speed process,
as this is produced by partial sums of the stationary
ergodic process of log-relative speeds. 
We leave this question to further work, but note that
the approach to this will depend a great deal on whether or not
the scattering process itself is point-recurrent.

Central limit behaviour for the scattering process itself
is complicated by the fact that in the true \SIRSN case \(\gamma>2\)
the times between scattering have statistics depending monotonically
on the current speed (see remarks earlier in this section).
However it may be possible to formulate the process as
being approximated by a constructed process
based on a Brownian motion,
using the coupling techniques of \citet{KendallWestcott-1987}.

It has been noted that the \SIRSNRRF defined here
cannot exist on high-dimensional \SIRSN (with \(\gamma>d>2\)),
for the simple reason that lines of Poisson line processes
in space of dimension \(3\) or higher will
almost surely never intersect.
However it \emph{does} make sense to ask whether this construction
can be generalized to line patterns in \(3\)-space
formed by a Poisson process of planes. 
However
it would first be necessary to extend the results
of \cite{Kendall-2014c} and \citet{Kahn-2015} to this situation.
Finally, it would be an interesting exercise to 
establish similar results 
for Rayleigh random flights on
\citeauthor{Aldous-2012}'s 
(\citeyear{Aldous-2012})
binary hierarchy \SIRSN.

\bigskip
\noindent \textbf{Acknowledgements.} 
The author acknowledges the support of 
the Isaac Newton Institute for Mathematical Sciences, Cambridge, 
under EPSRC grant EP/K032208 (``Random Geometry'' programme),
by the Alan Turing Institute under EPSRC grant EP/N510129,
and also by EPSRC grants EP/K013939 and EP/R022100 for the author's research.

This is a theoretical research paper and, as such, no new data were created during this study.

   \bibliographystyle{plainnat}
   \bibliography{SIRSN-RRF}

%

\end{document}